\documentclass[a4paper,10pt]{amsart}
\usepackage{amssymb}
\usepackage{graphicx}
\usepackage{mathtools}
\usepackage{slashbox,booktabs}

\usepackage{hyperref}
\hypersetup{
    colorlinks=true,
    linkcolor=black,
    citecolor=black,
    filecolor=black,
    urlcolor=black,
}

\newtheorem{theorem}{Theorem}[section]
\newtheorem{lemma}[theorem]{Lemma}
\newtheorem{proposition}[theorem]{Proposition}
\newtheorem{corollary}{Corollary}

\theoremstyle{definition}

\theoremstyle{remark}
\newtheorem{remark}[theorem]{Remark}

\numberwithin{equation}{section}
\newcommand{\x}{{\mathbf{x}}}
\newcommand{\z}{{\mathbf{z}}}
\newcommand{\xx}{{\mathbf{y}}}
\newcommand{\dist}{{|\x-\xx|}}
\newcommand{\distn}{{|\x-\xx|^{n+2s}}}

\renewcommand{\ker}{\gamma_{\delta}}
\newcommand{\kernel}{\gamma}
\renewcommand{\S}{{S_{\delta}}}

\newcommand{\R}{\mathbb{R}}
\renewcommand{\L}{\mathcal{L}_{\delta}}

\newcommand{\LL}{A_{\varepsilon}}
\newcommand*{\LLe}[1]{{A_{#1}}}
\newcommand{\A}{{\A_{\delta}}}

\newcommand{\D}{\Omega}
\newcommand{\DI}{\Omega_{\delta}}
\newcommand{\DD}{{\D\cup\DI}}
\newcommand{\DL}{\widetilde{\Omega}}

\newcommand{\X}{X}
\newcommand{\V}{V}
\newcommand{\Vinf}{\V(\R^n)}
\newcommand{\Vd}{{\V^{\prime}}}

\newcommand{\VD}{\V} 
\newcommand{\VDd}{\VD^{\prime}}
\newcommand{\XD}{\X}

\newcommand{\HD}{H^s_{\D}(\R^n)}
\newcommand{\HDi}{H^{-s}_{\D}(\R^n)}

\newcommand{\vt}{\xi} 
\newcommand{\zz}{z} 
\renewcommand{\O}{\mathcal{O}}

\newcommand{\intd}{\int_{\DD}}
\newcommand{\N}{h}
\renewcommand{\d}{\mathop{}\!\mathrm{d}}

\newcommand*{\norm}[1]{{\|{#1}\|}}
\newcommand*{\norms}[1]{{\left|{#1}\right|}}
\newcommand*{\maxx}[1]{{({#1})^+}}

\newcommand*{\normc}[1]{\norm{#1}_{\V}}

\newcommand*{\dual}[1]{{\langle{#1}\rangle}}
\newcommand*{\duall}[1]{\langle{#1}\rangle}

\newcommand{\spann}{\mathrm{span}}
\newcommand{\supp}{\mathrm{supp}}
\newcommand{\diam}{\mathrm{diam}}

\newcommand*{\rr}[2]{{#1}e--{0#2}}

\begin{document}

\title[Nonlocal variational (in-)equalities]{Regularity and approximation analyses of nonlocal variational equality and inequality problems}

\author{Olena Burkovska and Max Gunzburger}
\address[Olena Burkovska and Max Gunzburger]{Department of Scientific Computing, Florida State University, 400 
Dirac Science Library, Tallahassee, FL 32306-4120, USA}

\email[O. Burkovska]{oburkovska@fsu.edu}
\email[M. Gunzburger]{mgunzburg@fsu.edu}

\thanks{Supported by the US Air Force Office of Scientific Research grant FA9550-15-1-0001.}

\begin{abstract}
We consider linear and obstacle problems driven by a nonlocal integral operator, for which nonlocal interactions are restricted to a ball of finite radius. These 
type of operators are used to model anomalous diffusion and, for a special choice of the integral kernels, reduce to the fractional Laplace operator on a bounded domain. By means of a nonlocal vector calculus we recast the problems in a weak form, leading to corresponding nonlocal variational equality and inequality problems. We prove optimal regularity results for both problems, including a higher regularity of the solution and the Lagrange multiplier. Based on the regularity results, we analyze the convergence of finite element approximations for a linear problem and illustrate the theoretical findings by numerical results.
\end{abstract}

\maketitle
\section{Introduction}

We consider the analysis and approximation of a linear diffusion problem and a related obstacle problem involving the nonlocal operator
\begin{equation}
-\L 
u(\x):=2\int_{B_{\delta}(\x)}(u(\xx)-u(\x))\kernel(\x,\xx)\d\xx,\quad\x\in\D,
\label{nonl_operator0}
\end{equation}
where the kernel  $\kernel(\x,\xx):\R^n\times\R^n\to\R$ is a non-negative symmetric function, $\D\subset\R^n$ is a bounded domain, and $B_{\delta}(\x)$ is a ball in $\R^n$ of a radius $\delta>0$ centered at $\x$. The operator $\L$ is nonlocal because the value of the function $u$ at a point $\x$ is defined by the contributions of $u$ at other points $\xx$ separated from $\x$ by a finite distance. The parameter $\delta$ determines the extent of the nonlocal interactions.

If nonlocal interactions become infinite ($\delta\to\infty$) and the kernel $\kernel=1/\norms{\x-\xx}^{n+2s}$ up to some scaling factor, the nonlocal operator $-\L$ reduces to the fractional Laplace operator $(-\Delta)^s$, $s\in(0,1)$, on a bounded domain; see~\cite{DELIAlaplacian}. In this case, the kernel $\kernel$ is singular and the integral~\eqref{nonl_operator0} is understood in the principal value sense. From the probabilistic point of view, the nonlocal diffusion operator~\eqref{nonl_operator0} is associated with a L\'evy jump-diffusion processes with $\delta$ related to the maximum length of possible jumps.

In addition to the fractional Laplace kernel, other kernels have been actively exploited in different applications, e.g., 
peridynamics~\cite{silling2000}, machine learning~\cite{rosasco2010}, and image analysis~\cite{buades2010,gilboa2007}. Furthermore, an extension 
to nonlocal convection diffusion models have also been 
investigated~\cite{conv_diff2017,tian2015}. 

Nonlocal diffusion operators also arise in mathematical finance applications such as option pricing; see, e.g.,~\cite{cont2004}. In this context, a nonlocal diffusion operator is used to model the behaviour of the log-asset price that is not exclusively driven by a Brownian motion, but instead follows a jump-diffusion process, e.g., the Merton model~\cite{merton1976}. 

The (non-)local obstacle problem is often associated with the optimal stopping time problem in stochastic control, see, e.g.,~\cite{bensoussan}, that, in the context of option pricing, is related to the pricing of American put options.

In this paper, we analyze the stationary linear and obstacle problem associated with the nonlocal operator $-\L$ and also analyze finite element discretizations of these problems. A main contribution of this work is to derive improved Sobolev regularity for the solution of both problems. Such results are necessary, e.g., for the error analysis of finite element approximation.

We first consider, for a given function $f$, the following linear nonlocal problem with homogeneous volume constraints:
\begin{subequations}
 \begin{align}
 -\L u(\x)&=f(\x), &&\mbox{for } \x\in\D,\\
 u(\x)&=0, &&\mbox{for } \x\in\DI,\label{nonl_linear_constr}
\end{align}\label{nonl_linear}\end{subequations}
where $\DI\subset\R^n\setminus\D$ is the $\delta$-neighborhood of $\D$. Volume constraints such as \eqref{nonl_linear_constr} (that is the nonlocal analog of the Dirichlet boundary condition imposed on $\partial\D$ in the local partial differential equation case) are essential for the well posedness of the nonlocal problem. The variational analysis and finite element approximation of the problem are possible due to the recently developed nonlocal vector 
calculus~\cite{Dunonlocal2012,Dunonlocal2013}.

For the problem \eqref{nonl_linear}, we restrict attention to the kernel related to the fractional Laplacian. The Sobolev regularity of the fractional Laplace problem on bounded smooth domains in terms of the Sobolev regularity of the data is considered in \cite{grubb2015} using H\"ormander's theory for pseudo-differential operators. On less smooth domains and under stronger regularity assumption on the data, similar regularity estimates in (weighted) Sobolev spaces have been established in~\cite{acosta2017} by building upon H\"older regularity results for the fractional Laplace problem~\cite{ros_oton2014}. However, the corresponding results for nonlocal operators~\eqref{nonl_operator0} related to the \emph{truncated} fractional Laplacian are not available in the literature although they are necessary for the finite element analysis; see, e.g.,~\cite{Dunonlocal2012}. Our aim is to fill this gap and derive the corresponding regularity estimates for truncated kernels. We do so by relying upon results from~\cite{grubb2015} and by estimating the error terms arising from truncation; see Theorem~\ref{thm:regularity_linear}. We also develop the discretization of the problem by linear finite elements and derive the corresponding a priori error estimates.

In the second part of the paper, we consider the obstacle problem in which, in contrast to the linear model \eqref{nonl_linear}, the solution $u(\x)$ is additionally constrained by the given obstacle functional $\psi(\x)$ and solves the following set of inequalities:
\begin{subequations}
 \begin{align}
-\L u(\x)&\geq f(\x),
&&\mbox{for } \x\in\D,\\
u(\x)&\geq\psi(\x),  &&\mbox{for } \x\in\D,\\
\left(-\L u(\x)-f(\x)\right)\left(u(\x)-\psi(\x)\right)&=0, &&\mbox{for } \x\in 
\D,\\
u(\x)&=0, &&\mbox{for } \x\in\DI.
\end{align}\label{nonl_obstacle}\end{subequations}
The well-posedness of this problem and the convergence of the finite element approximation have been established in~\cite{guan2017} for the fractional Laplace and integrable kernels. We analyze this problem in a variational framework, considering instead of the primal variational inequality an equivalent variational saddle-point formulation. An additional Lagrange multiplier variable is introduced that physically can be interpreted as a contact force on the obstacle. The well posedness of the nonlocal (and local) problem require only low regularity assumptions on the Lagrange multiplier. In the local case, it is well known that for sufficiently regular right-hand sides and obstacle function, the regularity of the Lagrange multiplier can be improved together with the primal solution. {However, for the kernels under consideration in this work, similar results are not available}. In this paper, for a general class of kernels, we are able to derive improved regularity results for the Lagrange multiplier under certain regularity assumptions on the data; see Theorem~\ref{thm:regularity_lagrange}. For the proof, we follow the penalty approach developed for the analysis of the local variational inequalities~\cite{kinderlehrer80}, and derive the analogue of the the Lewy-Stampaccia dual estimates~\cite{lewy_1969} for the nonlocal case. Combining these with the new regularity estimates we derive for the problem \eqref{nonl_linear}, we are also able to show an improved regularity of the primal solution for the truncated fractional Laplace kernel. In addition, for this type of the kernel, we show the convergence of the nonlocal solution to the solution of the corresponding variational inequality for the fractional Laplace operator. 

We comment on other works related to the regularity of the obstacle problem for the fractional Laplacian. Similar Lewy-Stampaccia type estimates are obtained in~\cite{servadei} in the context of an abstract framework of (non-)local operators. However, the case of the truncated kernel is not covered by their analysis and higher regulary data is assumed together with an additional restriction on the power of the fractional Laplacian. The regularity of the obstacle problem measured in H\"older and Lipschitz spaces is extensively studied in~\cite{silvestre2007,caffarelli2017} by means of the equivalent representation of the problem in $n+1$ dimension using the Dirichlet-to-Neumann map~\cite{caffarelli2007}.  

The rest of the paper is structured as follows. In Section~\ref{sec:prelim}, we introduce the necessary function spaces and recall some preliminary results, which are used in the rest of the paper. Section~\ref{sec:ve} covers the analysis for the linear problem \eqref{nonl_linear} and the corresponding regularity results are derived in Section~\ref{subsec:ve_reg}. In Section~\ref{subsec:disc_linear}, we present a finite element discretization for the linear problem and derive the associated a priori error estimates. The nonlocal variational inequality formulation of problem \eqref{nonl_obstacle} is considered in~ Section~\ref{sec:vi}. For the fractional Laplace kernel, we show, in Section~\ref{subsec:vi_fl}, the convergence of the nonlocal solution to the fractional Laplacian solution. In Section~\ref{subsec:vi_reg}, we derive the improved regularity for the Lagrange multiplier and the corresponding primal solution. We discuss the finite element discretization of the nonlocal variational inequality in Section~\ref{subsec:disc_vi}. Numerical results that illustrate our theoretical findings are given in Section~\ref{sec:numerics} for one-dimensional linear and obstacle problems.

\section{Preliminaries}\label{sec:prelim}
Let $\D\subset\mathbb{R}^n$ be a bounded domain with a Lipschitz boundary. We introduce a truncated kernel $\ker:\R^n\times\R^n\to\R$ that is a non-negative symmetric function and that, and for all $\x\in\R^n$, satisfies the following conditions:
\begin{equation}
 \begin{cases}
   \ker(\x,\xx)\geq 0 &\forall\xx\in B_{\delta}(\x)\\
   \ker(\x,\xx)\geq\gamma_0>0 &\forall\xx\in B_{\delta/2}(\x)\\
   \ker(\x,\xx)=0 &\forall\xx\in\R^n\setminus B_{\delta}(\x),
  \end{cases}\label{kernel_prop1}
\end{equation}
where $\gamma_0$ is a positive constant and $B_{\delta}(\x)$ is
the ball of radius $\delta$ centered at $\x$:
\begin{equation*}
 B_{\delta}(\x):=\{\xx\in\R^n\colon\dist\leq\delta\}.
\end{equation*}
These conditions imply that the nonlocal interactions are limited to a ball of a radius $\delta$, referred as the interaction radius. For $\delta>0$, we define an interaction domain $\DI\subset\mathbb{R}^d\setminus\D$ corresponding to $\D$ as follows:
\[
\DI:=\{\xx\in\mathbb{R}^d\setminus\D\colon \ker(\x,\xx)\neq 0,\ \ \x\in\D 
\}.
\] 
In terms of these notations, we can express the operator~\eqref{nonl_operator0} as 
\begin{equation}
-\L 
u(\x):=2\int_{\DD}(u(\xx)-u(\x))\ker(\x,\xx)\d\xx,\quad\x\in\D.
\label{nonl_operator}
\end{equation}

Because the kernel $\ker(\x,\xx)$ determines the qualitative nature of the solution problem, the model formulations~\eqref{nonl_linear} and \eqref{nonl_obstacle} can  be used to cover a large class of problems. We analyze a several examples for the kernels. 

\subsection{Kernels}

Although we often focus on the kernels $\ker$ defined in Case 1 below that are related to the fractional Laplace operator, we also consider other possible choices for $\ker$.

{\it{Case 1.}} ({\em Fractional Laplacian-type kernels.}) For $s\in(0,1)$, $\delta>0$, and some positive constants $\gamma_{1},\gamma_2>0$, consider kernels which are proportional to $1/\distn$, namely, for all 
$\x\in\DD$
\begin{equation}
 \frac{\gamma_{1}}{\distn}\leq\ker(\x,\xx)\leq\frac{\gamma_2}
{ \distn } \quad\mbox{for }\xx\in B_{\delta}(\x).\label{kernel_prop2}
\end{equation}
As an example, consider the kernel
\begin{equation}
 \ker(\x,\xx)=\begin{dcases}
 \frac{\sigma(\x,\xx)}{\distn}, &{\xx\in B_{\delta}(\x)},\\
 0, &\text{otherwise}
 \end{dcases}\label{kernel_case1}
\end{equation}
with $\sigma(\x,\xx)$ a symmetric function that is bounded from below and above by positive constants. For $\delta=+\infty$, as a particular instance of~\eqref{kernel_case1}, we consider the kernel corresponding to the fractional Laplace operator~\eqref{FL_operator}:
\begin{equation}
\kernel_{\infty}(\x,\xx)=\frac{c_{n,s}}{2|\x-\xx|^{n+2s}},\quad c_{n,s}=\frac{2^{2s}s\Gamma(s+\frac{n}{2})}{\pi^{n/2}\Gamma(1-s)},
\label{eq:kernel_fractional} 
\end{equation}
where
$\Gamma(\cdot)$ denotes the Gamma function. Then, the operator $-\L$ corresponds to the integral definition of the fractional Laplace operator
\begin{equation}
 (-\Delta)^s 
u(\x):=c_{n,s}\int_{\R^n}\frac{u(\x)-u(\xx)}{\dist^{n+2s}}\d\xx,\quad 
0<s<1.\label{FL_operator}
\end{equation}
The fractional Laplace Poisson problem on bounded domains with homogeneous volume constraints takes the form
\begin{subequations}
 \begin{align}
 (-\Delta)^su(\x)&=f(\x)&&\mbox{for }\x\in\Omega,\\
 u(\x)&=0&&\mbox{for }\x\in\R^{n}\setminus\Omega.
\end{align}\label{FL_linear}\end{subequations}
In~\cite{DELIAlaplacian} the convergence of the nonlocal solution of~\eqref{nonl_linear} to the solution of the fractional Laplacian problem~\eqref{FL_linear} is shown.

{\it Case 2.} ({\em Square-integrable kernels.}) There exist constants $\gamma_3$, $\gamma_4>0$ such that
\begin{subequations}
\begin{align}
\gamma_3\leq\int_{(\DD)\cap 
B_{\delta}(\x)}\ker(\x,\xx)\d\xx,\quad
\int_{\DD}\ker^2(\x,\xx)\d\xx\leq\gamma_4^2, \quad\forall\x\in\D.
\end{align}\label{kernel_case2} 
\end{subequations}

{\it Case 3.} ({\em Peridynamic-type kernel.}) 
There exist constants $\gamma_5$, $\gamma_6>0$ such that
\begin{equation}
 \frac{\gamma_{5}}{\dist}\leq\ker(\x,\xx)\leq\frac{\gamma_6}
{ \dist} \quad\mbox{for }\xx\in B_{\delta}(\x).\label{kernel_prop3}
\end{equation}
A an example, consider the kernel 
\begin{equation}
 \ker(\x,\xx)=\begin{dcases}
 \frac{\xi(\x,\xx)}{\dist}, &{\xx\in B_{\delta}(\x)},\\
 0, &\text{otherwise}
 \end{dcases}\label{kernel_case3}
\end{equation}
with $\xi(\x,\xx)$ a symmetric function bounded from below and above by positive constants. The kernel~\eqref{kernel_case3} is integrable for $n>1$ and square-integrable for $n>2$ so that in the latter case, the kernel \eqref{kernel_prop3} is an example of the Case 2 type kernels.

Note that peridynamics is a nonlocal continuum model for solid mechanics featuring kernels with this type singularity, but, of course, with vector-valued displacement functions to be solved for. However, the scalar case we consider here and the vector case are entirely similar with respect to the regularity and other features of their solutions. 

\subsection{Nonlocal function spaces}
We introduce the function spaces used in this work and review some 
of the important properties of these spaces. 

For a general open set $\DL\subset\R^n$, we denote by $L^2(\DL)$ the standard space of square integrable functions on $\DL$. The fractional Sobolev space for $s\in(0,1)$ is defined as
\begin{equation*}
 H^s(\DL):=\{v\in L^2(\DL)\colon |v|_{H^s(\DL)}<\infty\}
\end{equation*}
with Gagliardo seminorm
\begin{equation*} 
|v|_{H^s(\DL)}^2:=\int_{\DL}\int_{\DL}\frac{|v(\x)-v(\xx)|^2}{\distn}
\d\xx\d\x.
\end{equation*}
The space $H^s(\DL)$ is a Hilbert space that is endowed with the norm 
\begin{equation*}
 \norm{v}_{H^s(\DL)}=\norm{v}_{L^2(\DL)}+\norms{v}_{H^s(\DL)}.
\end{equation*}

For $\sigma>1$ not an integer, we define $H^\sigma(\DL)$, $\sigma=m+s$, $m\in\mathbb{N}$, $s\in (0,1)$, as
\begin{equation*}
 H^\sigma(\DL):=\{v\in H^m(\DL)\colon |D^\alpha v|\in{H^s(\DL)},\;\mbox{for } 
|\alpha|=m\},
\end{equation*}
which is equipped with the norm
\begin{equation*} 
\norm{v}_{H^\sigma(\DL)}=\norm{v}_{H^m(\DL)}+\sum_{|\alpha|=m}\norms{D^\alpha v}_{
H^s(\DL) } .
\end{equation*}
Additionally, for $\DL$ such that $\overline{\D}\subset\DL$ and {$s>0$}, we define the space incorporating the volume constraints given by
\begin{equation*}
 H_{\D}^s(\DL):=\{v\in H^s(\DL)\colon v=0\;\mbox{on }\ \DL\setminus\D\},
\end{equation*}
that is endowed with the norm of $H^s(\overline{\D})$, i.e., $\norm{v}_{H_{\D}^s(\DL)}=\norm{v}_{H^s(\DL)}$.

We define the restriction operator $r_{\DL}:H^s_\D(\R^n)\to H_\D^s(\DL)$ by $r_{\DL}u=u\big|_{\DL}$ and the extension operator $e_{\DL}:H^s_\D(\DL)\to H^s_\D(\R^n)$ by
\begin{equation*}
e_{\DL}u=\begin{cases}
          u\;&\mbox{on }\DL,\\
          0\;&\mbox{on }\R^n\setminus\DL,
         \end{cases}
\end{equation*}
i.e., continuation of $u$ by zero outside~$\DL$.
The restriction and extension operators are linear continuous mappings, and {$H^s_\D(\DL)$ and $H^s_\D(\R^n)$ are isomorphic}.

For variational analyses of the problems~\eqref{nonl_linear} and \eqref{nonl_obstacle}, we define the bilinear form for $\ker$ satisfying~\eqref{kernel_prop1} and defined in one of 
{Case 1}, {Case 2}, or {Case~3} as
\begin{align}
a(u,v):=
\intd\intd\left(u(\x)-u(\xx)\right)\left(v(\x)-v(\xx)\right)\ker(\x , 
\xx)\d\xx\d\x.\label{bil_form}
\end{align}
The associated energy and constrained energy spaces are defined as 
\begin{align*}
\XD:=\{v\in L^2(\DD)\colon\; {a}(u,v)<\infty\}
\quad \mbox{and}\quad
\VD:=\{v\in \XD\colon\;  
v=0 \;\;\text{ a.e. on }\DI\}.
\end{align*}
Additionally, we have the constrained $L^2$-space
$$
L_{\D}^2(\DD):=\{v\in L^2(\DD)\colon\;  v=0 \;\;\text{ a.e. on }\DI\},
$$
{which is isometrically isomorphic to $L^2(\D)$}. For {Case~1} and $s\in (0,1)$, the nonlocal space $\XD$ is equivalent to the fractional Sobolev space $H^s(\DD)$ and $\V$ is equivalent to $H^s_\D(\DD)$. For {Case 2} and {Case 3} ($n>2$) the space $\X$ is equivalent to $L^2(\DD)$ and $V$ to $L_{\D}^2(\DD)$; see~\cite{Dunonlocal2012}. {For {Case~3}  ($n=1$), the energy space $\V$ is not equivalent to any Sobolev space, however it is a separable Hilbert space and is a strict subspace of $L^2(\DD)$.} 

Hence, $\X$ and $\VD$ are Hilbert spaces equipped with the inner product and norms
$$(u,v)_{\VD}:=a(u,v),\quad\quad\norm{v}^2_{\VD}:=a(v,v),\quad 
\norm{v}^2_{\X}=\norm{v}^2_{L^2(\DD)}+\norm{v}^2_\V.$$ 
For {Case~1} and some positive constants $C_1$ and $C_2$ and $s\in(0,1)$, we have the norm equivalence \cite{Dunonlocal2012}:
\begin{equation}
 C_1\norm{v}_{H^s(\DD)}\leq\norm{v}_{\VD}\leq 
C_2\norm{v}_{H^s(\DD)}\quad\quad\forall v\in\VD. \label{spaces_equiv}
\end{equation}
Moreover, for {Case 1} the space $V$ is equivalent to $H^s_\D(\R^n)$, $s\in(0,1)$, and that the following norm equivalence
\begin{equation}
C_3\norm{v}_{H_\D^s(\R^n)}\leq\norm{v}_{\VD}\leq 
C_4\norm{v}_{H_\D^s(\R^n)}\quad\forall v\in\VD\label{spaces_equiv_Rn}
\end{equation}
holds. This also implies that by means of the extension and restriction operators, we can always extend $u\in\V$ to $\widetilde{u}:=e_{\DD}u\in H^s_{\D}(\R^n)$, and vise-versa, for any $u\in H_\D^s(\R^n)$ we can restrict it to $\widetilde{u}:=r_{\DD}u\in\V$. Therefore, by an abuse of notation, we often omit the notation of these operators, and simply write $u$ if it is clear from the context.

We state another important result from~\cite{Dunonlocal2012} for the nonlocal space $\VD$.

\begin{lemma}[Nonlocal Poincar\'e inequality]
Let a kernel $\ker$ satisfy~\eqref{kernel_prop1} and defined in either {Case 1}, {Case 2} or {Case 3}. 
Then there exist a  constant $C_P>0$ such that the following holds
\begin{equation}
 \norm{v}_{L^2(\DD)}\leq C_P\norm{v}_{\VD}\quad\forall 
v\in\VD.\label{poincare}
\end{equation}\end{lemma}
We denote by $\VDd$ the dual space of $\VD$, and by 
$\dual{\cdot,\cdot}$ the extended {$L_\D^2(\DD)$} duality pairing between these spaces. For 
any 
$f\in\VDd$ we define the dual norm as 
\begin{equation*}
\norm{f}_{\VDd}:=\sup_{\substack{v\in\VD\\ v\neq 
0}}\frac{\int_{\D}f \ v\d \x}{\norm{v}_{\VD}}.
\end{equation*}
{In a similar way, we define $H^{-s}(\DL)=(H^s_\D(\DL))^{\prime}$, where, $\DL$ is such that $\overline{\D}\subset\DL$.}

\section{Linear nonlocal problem} \label{sec:ve}
Using the nonlocal Green's first identity~\cite{Dunonlocal2012}, we obtain the following weak formulation of the linear nonlocal problem~\eqref{nonl_linear}: For a given $f\in\VDd$, find $u\in\VD$ such that 
\begin{equation}
a(u,v)=\dual{f,v}\quad\forall v\in\VD.
 \label{nonl_linear_var}
\end{equation}
By the Lax-Milgram theorem, the problem~\eqref{nonl_linear_var} admits a unique solution. Moreover, there exist a constant $C>0$ such that solution satisfies
\begin{equation}
 \norm{u}_{\VD}\leq C\norm{f}_{\VDd}.\label{bound_f}
\end{equation}
Taking into account the equivalence of the nonlocal energy and fractional Sobolev spaces~\eqref{spaces_equiv}, the estimate~\eqref{bound_f} {for Case~1} implies that for any {$f\in H_\D^{-s}(\DD)$,}
\begin{equation}
 \norm{u}_{H_\D^s(\DD)}\leq C\norm{f}_{H_\D^{-s}(\DD)}\quad\mbox{for 
} 0<s<1.\label{bound_fs}
\end{equation}
For all $u, v\in\VD$ and $\delta>0$, we introduce the linear bounded operator $A_{\delta}:\VD\to\VDd$ associated with the bilinear form $a(\cdot,\cdot)$ given by 
\begin{equation}
\dual{A_{\delta}u,v}=a(u,v)\quad \forall 
u,v\in\V. \label{eq:operator_A}
\end{equation}

{For some cases}, we can consider the special case $\delta=+\infty$; the corresponding integral kernels are denoted by $\kernel_{\infty}$. Then, $\DI=\R^n\setminus\D$,
$\V=\Vinf$, and we introduce the bilinear form $a_\infty:\V\times\V\to\R$ given by
\begin{equation}
a_{\infty}(u,v)
=\int_{\R^n}\int_{\R^n}
\left(u(\x)-u(\xx)\right)\left(v(\x)-v(\xx)\right)\kernel_{\infty}(\x , 
\xx)\d\xx\d\x \quad\forall v\in\V,\label{eq:bil_a_inf}
\end{equation}
and the associated operator $A_{\infty}:\V\to\V^{\prime}$ defined as 
$\dual{A_{\infty}u,v}=a_{\infty}(u,v)$. {In order to ensure that~\eqref{eq:bil_a_inf} is well-defined, we restrict attention to {Case~1}, and {Case~2}, under an additional inegrability assumption on the kernel $\kernel_{\infty}$, specifically radial kernels with $\kernel_{\infty}(|\cdot|)\in L^1(\R^n\setminus B_\delta(0))$ for some $\delta>0$.} 

The following statement establishes a useful relation between the nonlocal operators $A_{\delta}$ and $A_{\infty}$ corresponding to kernels $\kernel_{\infty}$.

\begin{proposition}\label{prop:LP_nonlocal}
For $\delta>0$, assume the function $\ker$ is radial, i.e., $\ker(\x,\xx)=\ker(|\x-\xx|)$, and satisfies~\eqref{kernel_prop1} and, furthermore, we assume that $\kernel_{\infty}(
|\cdot|)\in L^1(\R^n\setminus B_\delta(0))$ for some $\delta>0$. Then, for $u\in\V$ and $\delta>\diam|\D|$ we have that
\begin{equation}
 {\dual{A_{\infty}u,v}=\dual{A_{\delta}u,v}+C(\delta,n)(u,v)_{L^2(\D)} ,}
\quad\forall v\in\V,\label{FL_nonl_relation}
\end{equation}
where $C(\delta,n)=2\int_{\R^n\setminus B_{\delta}(0)}\kernel_{\infty}(|\z|)\d\z.$
In addition, for {Case 1} with $\sigma(\x,\xx)=c_{n,s}/2$, that corresponds to the fractional Laplace kernel, $C(\delta,n)$ can be computed as
\begin{equation}
 C(\delta,n)=\frac{c_{n,s}\pi^{n/2}}{\Gamma({n}/{2})\delta^{2s}s}. 
\label{eq:equivalence_constant1}
\end{equation}
\end{proposition}
\begin{proof}
For $\delta>0$, we define the strip $\S:=\{(\x,\xx)\in\R^{2n}\colon \dist\leq\delta\}$. For all
$u,v\in\V$, we let $U(\x,\xx):=(u(\xx)-u(\x))(v(\xx)-v(\x))$ and consider
\begin{equation*} 
\dual{A_{\infty}u,v}
=\int_{\S}U(\x,\xx)\kernel_{\infty}(\x,\xx)\d(\xx,\x)
+\int_{R^{2n}
\setminus\S}U(\x,\xx)\kernel_{\infty}(\x,\xx)\d(\xx,\x) 
=I_1+I_2. \label{eq:I1_I2}
\end{equation*}
We consider $I_1$ and $I_2$ separately. We note, that $\supp(U)\subset(\R^n\times\overline{\D})\cup(\overline{\D} \times\R^n)$. Then, it is easy to verify that $\S\cap\supp(U)\subset(\overline{\DD})^2$, and hence
\begin{multline*}
I_1=\int_{\S}U(\x,\xx)\kernel_{\infty}(\x,\xx)\d(\xx,\x)=\int_{\S\cap(\DD)^2}
U(\x,\xx)\kernel_{\infty}(\x,\xx)\d(\xx,
\x\\
=\int_{\DD}\int_{\DD}
(u(\xx)-u(\x))(v(\xx)-v(\x))\ker(\x,\xx)\d\xx\d\x
=\dual{A_{\delta}u,v
} .
\end{multline*}
Invoking the symmetry with respect to $\x$ and $\xx$ and the integrability of $\ker(\x,\xx)$, we express $I_2$ as
\begin{multline}
I_2=\int_{R^{2n}\setminus\S}U(\x,\xx)\kernel_{\infty}(\x,\xx)\d(\xx,
\x)\\
=2\int_{\D}u(\x)v(\x)\int_{\R^n\setminus B_{\delta}(\x)}\kernel_{\infty}(\x,\xx)\d\xx\d\x
-2\int_{\D}u(\x)\int_{\R^n\setminus B_{\delta}(\x)}v(\xx)\kernel_{\infty}(\x,\xx)\d\xx\d\x. 
\label{eq:I2}
\end{multline}
We notice that for $\delta\geq\diam|\D|$, the last term in~\eqref{eq:I2} vanishes due to the fact that for $\x\in\D$, $(\R^n\setminus B_{\delta}(\x))\cap\D=\emptyset$. Combining expressions for $I_1$ and $I_2$, we obtain~\eqref{FL_nonl_relation}.

For kernels $\kernel_{\infty}$ as in {Case 1} with $\sigma(\x,\xx)=\sigma$, we obtain
\begin{equation*}
\int_{\R^n\setminus B_{\delta}(\x)}\kernel_{\infty}(\x,\xx)\d\xx=\int_{\R^n\setminus 
B_{\delta}(0)}\frac{\sigma}{|\z|^{n+2s}}\d\z=\frac{\sigma\omega_n}{2s\delta^{2s}},
\end{equation*}
where $\omega_n$ is the $(n-1)$-dimensional measure of the unit sphere embedded 
in dimension $n$ and $\omega_n=\frac{2\pi^{n/2}}{\Gamma(n/2)}$. Taking $\sigma=c_{n,s}/2$, we conclude the proof.
\end{proof}
\begin{remark}
From the implementation point of view,~\eqref{FL_nonl_relation} 
provides a very useful relation between, e.g., the nonlocal and fractional 
Laplace operators, as it allows to assemble only one of those matrices and then 
only
subtract or add a mass matrix term correspondingly. 
\end{remark}

\subsection{Regularity study for the linear nonlocal problem}\label{subsec:ve_reg}

We derive a regularity result for the linear nonlocal problem~\eqref{nonl_linear} with the  Case 1 kernels.

Consider the following weak formulation of the problem~\eqref{FL_linear}: for given $f\in{\HDi}$, find $u\in\HD$ such that
\begin{equation} 
\frac{c_{n,s}}{2}\int_{\R^n}\int_{\R^n}\frac{u(\xx)-u(\x)}{\distn}
(v(\xx)-v(\x))\d\xx\d\x=\int_{\D}fv\d\x\quad\forall 
v\in\HD.\label{FL_linear_var}
\end{equation}
We recall the regularity result for the fractional Laplace problem on  bounded smooth domains stated in~\cite{grubb2015} in terms of H\"ormander $\mu$-spaces; {see also, e.g.,~\cite{bonito2017}, for the reinterpretation of this result}. For an earlier result in a less general framework; see~\cite{visik1967}.

\begin{theorem}[Fractional Laplace problem on bounded domains.]\label{regularity_grubb}
Let $\D$ be a domain with $C^{\infty}$ boundary $\partial\D$, and for $s\in(0,1)$, let $f\in {H^r(\D)}$, $r\geq-s$ and let $u\in H_{\D}^s(\R^n)$ be the solution of the fractional Laplace problem~\eqref{FL_linear_var}. Then, we have the regularity estimate  
\begin{equation}
|u|_{H_\D^{s+\alpha}(\R^n)}\leq 
C\norm{f}_{{H^r(\D)}},\label{FL_regularity}
\end{equation}
where $\alpha=\min\{s+r,1/2-\varepsilon\}$, with arbitrarily small $\varepsilon>0$.
\end{theorem}
{We note that independently of the smoothness of the right-hand side $f$, we cannot expect the solution $u$ to be any smoother than $H_{\D}^{s+1/2-\varepsilon}(\R^n)$; see, e.g.,~\cite{bonito2017} for a counterexample.  }

Before extending this result to the linear nonlocal problem~\eqref{nonl_linear}, we first provide an auxilliary result that will be useful for the main proof of Theorem~\ref{thm:regularity_linear}.

\begin{proposition}\label{prop:regularity_g}
Let $\DL\subset\R^n$ be an open set, let $\omega\in H_\D^{r}(\DL)$ for $r\geq 0$, where for $r=0$, $H_\D^0(\DL)\equiv L^2(\D)$, and let
\begin{equation}
g(\x):=\int_{\DL\setminus 
B_{\delta}(\x)}\frac{\omega(\xx)}{\distn}\d\xx \quad\mbox{for    
}\x\in\DL,\quad 0<s<1.
\end{equation}
Then, $g\in H_\D^{r}(\DL)$, $r\geq 0$, and the estimate
\begin{equation}
\norm{g}_{H^r(\DL)}\leq C\norm{\omega}_{H^{r}(\DL)}\label{eq: g_regularity}
\end{equation}
holds, where $C=\frac{\pi^{n/2}}{\Gamma(n/2)\delta^{2s}s}$.
\end{proposition}

\begin{proof}
Because $g=0$ for $\delta>\diam|\DL|$, we need only consider $\delta<\diam|\DL|$. For all 
$\x,\xx\in\R^n$ we have the estimate
\begin{align*}
\left|g(\xx)-g(\x)\right|&=\left|\int_{\DL\setminus 
B_{\delta}(\xx)}\frac{\omega(\z)}{|\xx-\z|^{n+2s}}\d\z-
\int_{\DL\setminus 
B_{\delta}(\x)}\frac{\omega(\z)}{|\x-\z|^{n+2s}}\d\z\right|\\
&\leq\sqrt{\int_{\DL\setminus 
B_{\delta}(0)}\frac{(\omega(\xx-\z)-\omega(\x-\z))^2}{|\z|^{n+2s}}\d\z}\sqrt{
\int_{\DL\setminus B_{\delta}(0)}\frac{\d\z}{|\z|^{n+2s}}}.
\end{align*}
Taking into account that the last term can be computed exactly, for $r>0$, we obtain
\begin{align*} 
\norms{g}_{H^{r}(\DL)}^2&=\int_{\DL}\int_{\DL}\frac{(g(\xx)-g(\x))^2}{
|\xx-\x|^ { n+2r}}\d\xx\d\x\\
&\leq
\frac{\pi^{n/2}}{\Gamma(n/2)\delta^{2s}
s}\int_{\DL}\int_{\DL}\int_{\DL\setminus 
B_{\delta}(0)}\frac{(\omega(\xx-\z)-\omega(\x-\z))^2}{ |\xx-\x|^
{ n+2r}{|\z|^{n+2s}}}\d\z\d\xx\d\x\\
&=\frac{\pi^{n/2}}{\Gamma(n/2)\delta^{2s}
s}\int_{\DL\setminus 
B_{\delta}(0)}\frac{1}{|\z|^{n+2s}}\d\z\int_{\DL}\int_{\DL}\frac{
(\omega(\xx)-\omega(\x))^2}{
|\xx-\x|^ { n+2r}}\d\xx\d\x\\
&=\left(\frac{\pi^{n/2}}{\Gamma(n/2)\delta^{2s}
s}\right)^2\norms{\omega}_{H^r(\DL)}^2.
\end{align*}
Following the same steps as above we obtain the corresponding $L^2$-norm estimate, and conclude the proof.
\end{proof}

We now prove the main result of this section.

\begin{theorem}[Nonlocal problem with truncated interactions]\label{thm:regularity_linear}
Let $\D$ be a domain with $C^{\infty}$ boundary $\partial\D$, and let $f\in  H^r(\D)$, $r\geq 0$, and let $u\in\VD$ be the solution of~\eqref{nonl_linear_var} with the Case 1 kernel with $\sigma(\x,\xx)=\sigma$ and $\delta>0$. Then, for a positive constant $C>0$, we have the regularity estimates
\begin{equation}
|u|_{H_\D^{s+\alpha}(\R^n)}\leq C\norm{f}_{{H^r(\D)}},
\end{equation} \label{nonl_regularity}
where $\alpha=\min\{s+r,1/2-\varepsilon\}$ for some arbitrarily small $\varepsilon>0$. We note that by writing $u$ in~\eqref{nonl_regularity}, we mean its extension $e_{\DD}u$.
\end{theorem}
\begin{proof}
Without loss of generality, we conduct the proof for $\sigma={c_{n,s}/2}$. Following the steps of the proof of Proposition~\ref{prop:LP_nonlocal}, we can write for all $v\in H^s_{\D}(\R^n)$
\begin{multline}
\frac{c_{n,s}}{2}\int_{\R^n}\int_{\R^n}\frac{u(\xx)-u(\x)}{\distn}
(v(\xx)-v(\x))\d\xx\d\x\\
=\frac{c_{n,s}}{2}\int_{\DD}\int_{\DD\cap 
B_{\delta}(\x)}\frac{u(\xx)-u(\x)}{\distn}
(v(\xx)-v(\x))\d\xx\d\x\\
+C(\delta,n)\int_{\D}u(\x)v(\x)\d\x
-c_{n,s}\int_{\D}g(\x)v(\x)\d\x,\label{eq:regularity_main}
\end{multline}
where $C(\delta,n)$ is defined in~\eqref{eq:equivalence_constant1} and 
$$g(\x):=\int_{
\D\setminus 
B_{\delta}(\x)}\frac{u(\xx)}{\distn}\d\xx.$$ 
Taking 
into account that for $f\in H^r(\D)$, $u\in\VD$ is a solution of~\eqref{nonl_linear_var}, we can then rewrite~\eqref{eq:regularity_main} as 
\begin{align}
\frac{c_{n,s}}{2}\int_{\R^n}\int_{\R^n}\frac{u(\xx)-u(\x)}{\distn}
(v(\xx)-v(\x))\d\xx\d\x=\int_{\D}Fv\d\x,\label{eq:regularity_2}
\end{align}
where 
\begin{align*}
\int_{\D}F v\d\x:=\int_{\D}f 
 v\d\x-C(\delta,n)\int_{\D}u  v\d\x
+c_{n,s}\int_{ \D}g  v\d\x
\end{align*}
or simply $F=f -C(\delta,n)u+c_{n,s}g$.
For $r=0$, i.e., $f\in L^2(\D)$, by using Proposition~\ref{prop:regularity_g},~\eqref{poincare}, and~\eqref{bound_f} we obtain that $F\in L^2(\D)$ and
\begin{align*}
 \norm{F}_{L^2(\D)}\leq\norm{f}_{L^2(\D)}+2C(\delta,n)\norm{u}_{L^2(\D)}\leq 
C\norm{f}_{L^2(\D)}.
\end{align*}
Now, applying Theorem~\ref{regularity_grubb} to~\eqref{eq:regularity_2} with the right-hand side $F$, we obtain that $u\in H_\D^{s+\alpha_0}(\R^n)$, where $\alpha_0=\min\{s,1/2-\varepsilon\}$ with $\varepsilon>0$ and moreover, we have that
\begin{equation}
 \norms{u}_{H^{s+\alpha_0}_\D(\R^n)}\leq 
C\norm{f}_{L^2(\D)}.\label{eq:regularity_temp1}
\end{equation}
for $C>0$. Following a boot-strapping technique, let now $f\in H^r(\D)$ for $r>0$, then we obtain that $F\in H^{\beta_1}(\D)$, $\beta_1=\min\{r,s+\alpha_0\}$, and by Theorem~\ref{regularity_grubb}, $u\in H^{s+\alpha_1}_\D(\R^n)$, where $\alpha_1=\min\{s+r,3s,1/2-\varepsilon\}$, $\varepsilon>0$. Using repeatedly the regularity result~\eqref{eq:regularity_temp1} from the previous step, Proposition~\ref{prop:regularity_g}, and the Sobolev imbedding theorem (see, e.g.,~\cite{grisvard}), we obtain the regularity estimate 
\begin{equation*}
\norms{u}_{H^{s+\alpha_1}_\D(\R^n)}\leq 
C\norm{F}_{H^{\beta_1}(\D)}\leq 
C\norm{f}_{H^{r}(\D)},
\end{equation*}
for some constant $C>0$. By iterating the previous arguments, we obtain that, for $m\in\mathbb{N}$, $F\in H^{\beta_m}(\D)$, $\beta_m=\min\{r,2ms,s+1/2-\varepsilon\}$, $u\in H^{s+\alpha_m}_\D(\R^n)$, $\alpha_m=\min\{s+r,(1+2m)s,1/2-\varepsilon\}$, and,
\begin{equation}
 \norms{u}_{H^{s+\alpha_m}_\D(\R^n)}\leq C\norm{F}_{H^{\beta_m}(\D)}\leq 
C\norm{f}_{H^r(\D)},
\end{equation}
 for $C>0$. Clearly, for $m$ large enough, we obtained that $\beta_m=\min\{r,s+1/2-\varepsilon\}$, 
and $\alpha_m=\min\{s+r,1/2-\varepsilon\}$. Denoting by $\alpha=\alpha_m$, we have obtain \eqref{nonl_regularity}. 
\end{proof}

To the best of our knowledge, the regularity of the solution of the fractional Laplace problem~\eqref{FL_linear} on less regular, e.g., Lipschitz, domains remains an active field of research. In \cite{acosta2017}, similar results as in Theorem~\ref{regularity_grubb} in terms of (weighted) Sobolev spaces are obtained for Lipschitz domains under the assumption of H\"older regularity on the data. However, these higher regularity assumptions on the data are not applicable to the regularity of the solution of the variational inequality considered in Section~\ref{sec:vi}. 
\subsection{Discretization of the linear problem and a priori error estimates}\label{subsec:disc_linear}

Let $\D$ be a convex domain with $C^2$-boundary. We subdivide $\DD$ into a shape regular quasi-uniform triangulation $\{\mathcal{T}_\N\}_\N$; see, e.g.,~\cite{brennerscott,ciarlet}. We denote by $h$ the maximum diameter of the elements $K\in\mathcal{T}_\N$ and set $\overline{\D^h\cup\DI^h}=\cup_{K\in\mathcal{T}_\N}\overline{K}$ under the assumption that  $\D^h\subset\D$ and $|\D\setminus\D^h|\leq c h^2$, $|\DI\setminus\DI^h|\leq c h^2$, which can be realized by requiring boundary nodes to lie on the boundary of $\D$ and $\DI$, respectively.

We use the subspace $V_\N\subset\V$ of piecewise linear polynomials associated with $\mathcal{T}_h$ that satisfy a homogeneous Dirichlet volume constraint in $\DI^h$, i.e.,
\begin{equation}
{\V_\N=\{v_\N\in C^0({\overline{\DD}})\colon 
v_\N|_{K}\in\mathcal{P}_1(K)\; \forall K\in\mathcal{T}_h,\; v_\N=0\; \text{ on 
}\; {(\DD)\setminus\D^h}\}.}\label{eq:V_discrete}
\end{equation}
Let $\mathcal{J}_\N^m$ denote the set of all interior nodes of $\D$ and $\mathcal{J}_\N^l$ the set of all nodes. We can represent $\X_\N=\spann\{\phi_p, \;p\in\mathcal{J}_\N^l\}$ and $\V_\N=\spann\{\phi_p, \;p\in\mathcal{J}_\N^m\}$, where $\phi_p$ are the nodal Lagrange basis 
functions. 

The discrete formulation of the linear problem~\eqref{nonl_linear_var} becomes: find $u_\N\in\V_\N$ such that
\begin{equation}
{a(u_\N,v_\N)}=\dual{f,v_\N}\quad\forall v\in\V_\N.\label{nonl_linear_disc}
\end{equation}
The existence and uniqueness of the solution of~\eqref{nonl_linear_disc} 
directly follows from the 
{conformity of the discrete spaces $\V_\N$}.  In addition, the best approximation property holds~\cite{Dunonlocal2012}:
\begin{equation}
 \norm{u-u_\N}_\V\leq\min_{v_\N\in\V_\N}\norm{u-v_\N}_\V\to 0\;\;\text{as  
}{\N\to 0}.\label{best_approx}
\end{equation}

To obtain convergence estimates, we require that there exists a quasi-interpolation operator $\Pi_\N\colon 
H_\D^s(\R^n)\to V_\N$ such that
\begin{equation}
\norm{u-\Pi_\N u}_{H^s_\D(\R^n)}\leq 
Ch^{\alpha}\norm{u}_{H^{s+\alpha}_\D(\R^n)} \label{eq:interpol}
\end{equation}
for $0<\alpha<1/2$, and some positive constant $C$ independent of $h$. Then, exploiting the regularity of the solution and using the best approximation property~\eqref{best_approx}, we  obtain the following a priori error estimates for the finite elements discretization for the linear problem~\eqref{nonl_linear_var}. 

\begin{proposition}\label{prop:convergence_rates}
Let $\D$ be a domain with $C^{\infty}$ boundary $\partial\D$, and let for $f\in H^r(\D)$,  $r\geq 0$, $u\in\V$ be a solution of~\eqref{nonl_linear_var} with kernel $\ker$ satisfying~\eqref{kernel_prop1} and defined as in {\it Case~1} with $\sigma(\x,\xx)=\sigma$. Then for $s\in(0,1)$ there exists a constant $C>0$, independent of $h$, such that
\begin{subequations}
\begin{align}
 \norm{u-u_\N}_{H_\D^s(\R^n)}\leq 
Ch^{\alpha}\norm{f}_{H^r(\D)},\label{eq:convergence_ratesHS}\\
\norm{u-u_\N}_{L^2(\D)}\leq 
Ch^{\alpha+\beta}\norm{f}_{H^r(\D)},\label{eq:convergence_ratesL2}
\end{align} \label{eq:convergence_rates}
\end{subequations}
where $\alpha=\min\{s+r,1/2-\varepsilon\}$, and 
$\beta=\min\{s,1/2-\varepsilon\}$ for some $\varepsilon>0$.
\end{proposition}

\begin{proof}
The proof for the estimates in the energy norm~\eqref{eq:convergence_ratesHS} is directly obtained by combining~\eqref{best_approx},~\eqref{eq:interpol}, and the regularity estimates~\eqref{nonl_regularity}. The convergence estimate in the $L^2$-norm~\eqref{eq:convergence_ratesL2} is obtained by applying a standard Aubin-Nitsche duality argument together with the regularity results~\eqref{nonl_regularity}.
\end{proof}


\section{Nonlocal variational inequality}\label{sec:vi}

To derive the variational formulation of the obstacle problem~\eqref{nonl_obstacle} we introduce a set $\mathcal{K}$ of admissible solutions 
\begin{equation*}
\mathcal{K}:=\{u\in\VD\colon\; u\geq\psi\;\mbox{on }\D\}, 
\end{equation*}
for all $\psi\in \XD$ and assume that {$\psi(\x)\leq 0$ on $\DI$}. It is easy to see that $\mathcal{K}$ is closed, convex, and non-empty. Then, a weak formulation of~\eqref{nonl_obstacle} leads to the following  variational inequality problem: find $u\in\mathcal{K}$ such that for a given $f\in\VDd$ it holds
\begin{equation}
a(u,v-u)\geq\dual{f,v-u},\quad\forall v\in 
\mathcal{K}.\label{eq:vi}
\end{equation}
The problem admits a unique solution; see~\cite{glowinski,guan2017}.

By means of Lagrange multipliers, we restate the problem~\eqref{eq:vi} in an equivalent mixed formulation or saddle point form. 
Let the space $W:=\VDd$ and define a bilinear form $b:W\times\V\to\mathbb{R}$ as a duality pairing $b(\eta,v):=\dual{\eta,v}$. In $W$, we define the set $M$ (referred to as a the dual cone) as
\begin{equation}
M:=\{\eta\in W\colon\; b(\eta,v)\geq 0,\;v\in\VD, \;v\geq 
0\}.\label{eq:cone}
\end{equation}
Obviously, the bilinear form $b(\eta,v)$ is bounded and inf-sup stable on $W\times\V$, i.e.,
\begin{equation}
\beta=\inf_{\eta\in W,\eta\neq 0}\sup_{v\in\V, v\neq 
0}\frac{b(\eta,v)}{\norm{\eta}_W\normc{v}}\geq\beta_0>0.
\end{equation}
In the present setting, we have $\beta=1$. Then, the set $\mathcal{K}$ can be expressed equivalently as $\mathcal{K}=\{v\in\V\colon\; b(\eta,v)\geq b(\eta,\psi),\; \forall\eta\in 
M\}$, and we arrive in an equivalent saddle point formulation (see, e.g.,~\cite{kikuchi}):
find $u\in\VD$ and $\lambda\in M$ such that
\begin{subequations}
\begin{align}
a(u,v)-b(\lambda,v)&=\dual{f,v}, &&\forall v\in\VD,\\
b(\eta-\lambda,u)&\geq b(\eta-\lambda,\psi),&&\forall\eta\in M. 
\end{align}\label{eq:sp}\end{subequations}

\subsection{Variational inequalities with the fractional 
Laplacian}\label{subsec:vi_fl}

It is easy to see, that, similarly as in the linear case, for $\ker$ defined as in {Case~1} with $\sigma(\x,\xx)={c_{n,s}/2}$ in~\eqref{kernel_case1}, the solution of the nonlocal variational inequality~\eqref{eq:sp} converges to the solution of the variational inequality 
with the fractional Laplace operator as $\delta\to\infty$. Let
\begin{equation}
 M_{\infty}:=\{\eta\in H^{-s}_\D(\R^n)\colon\; b(\eta,v)\geq 0,\;v\in 
H_\D^s(\R^n), \;v\geq 
0\},
\end{equation}
where ${b(\cdot,\cdot)}$ is defined as the extended $L_{\D}^2(\R^n)$ duality pairing on $H^{-s}_\D(\R^n)\times H^{s}_\D(\R^n)$. Then, we consider the following variational inequality problem corresponding to the fractional Laplace operator~\eqref{FL_operator}: for $0<s<1$ find $u\in 
H^{s}_\D(\R^n)$ and $\lambda\in M_{\infty}$ such that
\begin{subequations}
\begin{align}
a_{\infty}(u,v)-b(\lambda,v)&=\dual{f,v}, &&\forall v\in 
H^{s}_\D(\R^n),\\
b(\eta-\lambda,u)&\geq 
b(\eta-\lambda,\psi),&&\forall\eta\in M_{\infty}, \label{eq:sp_fl_2}
\end{align}\label{eq:sp_fl}\end{subequations}
where $a_{\infty}(\cdot,\cdot)$ is defined in~\eqref{eq:bil_a_inf} with $\kernel_{\infty}$ defined in~\eqref{eq:kernel_fractional}, and {$\psi\in H^{s}(\R^n)$, {$\psi\leq 0$ on $\R^n\setminus\D$}.
Because the bilinear form $a(\cdot,\cdot)$ (or $a_{\infty}(\cdot,\cdot)$) defines an inner product on $\V$ (or $H^s_\D(\R^n)$), and $b(\cdot,\cdot)$ is inf-sup stable, the problem~\eqref{eq:sp} (or~\eqref{eq:sp_fl}) admits a unique solution. 

\begin{proposition}\label{prop:LP_VI_nonlocal}
Let $(u_{\infty},\lambda_{\infty})\in H_\D^s(\R^n)\times M_{\infty}$ be a solution pair of the variational inequality with the fractional Laplace operator~\eqref{eq:sp_fl}, and let $(u_{\delta},\lambda_{\delta})\in\V\times M$ denote the solution of the nonlocal variational inequality~\eqref{eq:sp} with $\ker$ satisfying~\eqref{kernel_prop1} and defined in~\eqref{kernel_case1} with $\sigma(\x,\xx)={c_{n,s}/2}$ and $\delta>\diam{|\D|}$. Then, 
\begin{subequations}
 \begin{align}
\norm{u_{\infty}-u_{\delta}}_{H_\D^s(\R^n)}\leq\frac{C_P}{C_3}C(\delta,n)\norm{
u_ { \infty}}_{L^2(\D)},
 \end{align}\end{subequations}
where the constants $C_P$, $C_3$, and $C(\delta,n)$ are defined in~\eqref{poincare},~\eqref{spaces_equiv_Rn}, and~\eqref{eq:equivalence_constant1}, respectively. Here, $C(\delta,n)\sim{\delta^{-2s}}$ and $C(\delta,n)\to 0$ as $\delta\to\infty$. 
\end{proposition}
\begin{proof}
From Proposition~\ref{prop:LP_nonlocal} we obtain that for all $v\in 
H_\D^s(\R^n)$
\begin{align*}
a(u_{\infty}-u_{\delta},v)+C(\delta,n)(u_{\infty},v)_{L^2(\D)}-b(\lambda_{\infty
}-\lambda_{\delta},v)=0.
\end{align*}
Setting $v:=u_{\infty}-u_{\delta}$ and using Cauchy-Schwarz inequality we obtain
\begin{align}
 \norm{u_{\infty}-u_{\delta}}_\V^2\leq 
C(\delta,n)\norm{u_{\infty}}_{L^2(\D)}\norm{u_{\infty}-u_{\delta}}_{L^2(\D)}
+b(\lambda_{\infty
}-\lambda_{\delta},u_{\infty}-u_{\delta}).\label{eq:vi_proof_1}
\end{align}
We note that the inequality constraint~\eqref{eq:sp_fl_2} can be expressed as
\begin{equation*}
 b(\lambda_{\infty},u_{\infty})=b(\lambda_{\infty},\psi),\quad\mbox{and }\quad
 b(\eta,u_{\infty})\geq b(\eta,\psi),\quad\forall\eta\in M_{\infty}.
\end{equation*}
This property is obtained by taking $\eta=0$ and $\eta=2\lambda_{\infty}$ in~\eqref{eq:sp_fl_2}. Analogous properties hold for the solution pair $(u_{\delta},\lambda_{\delta})\in\V\times M$ of~\eqref{eq:sp}. Then, taking into account the last properties and the equivalence of the dual cones $M_{\infty}$ and $M$ (in the sense, that for any $\eta\in M_{\infty}$ we have $\eta\in M$ and vice versa), we can estimate
\begin{align*}
 b(\lambda_{\infty}-\lambda_{\delta},u_{\infty}-u_{\delta})\leq 0.
\end{align*}
Combining the last estimate and using~\eqref{poincare}, the estimate~\eqref{eq:vi_proof_1} becomes
\begin{align*}
\norm{u_{\infty}-u_{\delta}}_\V\leq 
C_PC(\delta,n)\norm{u_{\infty}}_{L^2(\D)}.
\end{align*}
Finally, using the norm equivalence~\eqref{spaces_equiv_Rn} we obtain the necessary result.
\end{proof}

\begin{remark}
{We note that when $\psi=-\infty$, $\lambda_{\delta}=0$ and $\lambda_{\infty}=0$, and the variational inequality problem~\eqref{eq:sp} or~\eqref{eq:sp_fl} reduces to the linear problem~\eqref{nonl_linear_var} or~\eqref{FL_linear_var}, respectively. 
Then, the error estimates in Proposition~\ref{prop:LP_VI_nonlocal} also recover the error estimates, derived in~\cite{DELIAlaplacian} for the solutions of the corresponding linear problems. In fact, for both linear and variational inequality problems we obtain the same estimate and $C(\delta,n)\to 0$ as $\delta\to\infty$ at a rate $2s$.}
\end{remark}

\subsection{Regularity study for nonlocal variational inequalities}\label{subsec:vi_reg}

In the previous section, we posed a variational inequality problem~\eqref{eq:sp} under 
low regularity assumptions on the Lagrange multiplier $\lambda\in W=\Vd$. Now, we demonstrate that for a general class of kernels $\ker$
satisfying~\eqref{kernel_prop1}, and with sufficiently regular $\psi$ and $f$, we have that $\lambda\in L^2(\D)$. 

Moreover, for a class of kernels corresponding to the fractional Laplacian, we also prove an improved regularity result for the primal solution $u$ of~\eqref{eq:sp} by extending the improved regularity result for the solution of the linear fractional Laplace problem. We derive the regularity estimates for the nonlocal variational inequality~\eqref{eq:sp} by generalizing the Levy-Stampacchia estimates for local variational inequalities; see~\cite{kinderlehrer80}. {We note that similar results have been established for the fractional Laplacian in~\cite{servadei}, cf. also~\cite{musina2017}.  However, the case of truncated kernel is not covered by these analyses and the additional assumption $n>2s$ is imposed. Here, we give an independent derivation of the results for a general class of truncated kernels $\ker$. } 

First, we state some results from the theory of monotone operators that are useful for our analyses. 

\begin{lemma}[Minty Lemma,~{\cite[Chapter \textrm{III}, Lemma 
1.5]{kinderlehrer80}}]\label{lemma:minty}
Let $\mathcal{X}$ be a reflexive Banach space and $\mathbb{K}\subset\mathcal{X}$ be a closed convex set, and $A:\mathbb{K}\to\mathcal{X}^{\prime}$ be monotone and continuous on finite-dimensional subspaces. Then, $u\in\mathbb{K}$ satisfies 
$\duall{Au,v-u}_{\mathcal{X}^{\prime}\times\mathcal{X}}\geq 
0$, $\forall v\in\mathbb{K}$,
if and only if $u\in\mathbb{K}$ satisfies
$\duall{Av,v-u}_{\mathcal{X}^{\prime}\times\mathcal{X}}\geq 0$, $\forall v\in\mathbb{K}$.
\end{lemma}
In what follows, we assume that
\begin{equation}
f\in L^2(\D)\quad\mbox{and }\;
\maxx{-\L\psi-f}{\big|_\D}\in 
L^2(\D), \label{eq:f_psi}
\end{equation}
where $\maxx{\cdot}:=\max(\cdot,0)$, and, as before, $-\L$ is defined in~\eqref{nonl_operator} with kernels $\ker$ satisfying~\eqref{kernel_prop1}, and defined in either {Case~1}, {Case~2}, or {Case~3}. To derive a regularity result, we study an approximation of the variational inequality by a penalized problem. For any $\varepsilon>0$, introduce a sequence of penalty functions $\mathcal{H}_{\varepsilon}:\R\to[0,1]$, which are assumed to be uniformly Lipschitz continuous, non-increasing, and bounded, {$0\leq\mathcal{H}_{\varepsilon}(t)\leq 1$}. In particular, we can specify
\[
 \mathcal{H}_{\varepsilon}(t)=\begin{cases}
                    1 &\mbox{for }\; t\leq 0,\\
                    1-{t}/{\varepsilon}  &\mbox{for }\; 0\leq t\leq\varepsilon,\\
                    0 &\mbox{for }\; t\geq\varepsilon,
                   \end{cases}
\]
and consider the following penalized problem
\begin{subequations}
 \begin{align}
 -\L 
u_{\varepsilon}&= 
\maxx{-\L\psi-f}\mathcal{H}_{\varepsilon}(u_{\varepsilon}-\psi)+f &&\mbox{in 
} 
\D,\\
u_{\varepsilon}&= 0 &&\mbox{on } \D_I.
\end{align}\label{eq:penalized_vi}\end{subequations}
In a variational form, the above problem reads as follows: for a fixed $\varepsilon>0$, find $u_{\varepsilon}\in\VD$, such that for all $v\in\VD$
\begin{align} 
a(u_{\varepsilon},v)-\int_{\DD}\left(\maxx{-\L\psi-f}\mathcal{H}_{\varepsilon}
(u_{\varepsilon}-\psi)\right)v\d\x=\int_{\D}fv\d\x.\label{eq:var_penalized_vi}
\end{align}
Using the theory of nonlinear monotone operators, one obtains the existence and uniqueness of solutions of \eqref{eq:var_penalized_vi}.

\begin{lemma}\label{lemma:penalized}
Let  \eqref{eq:f_psi} be satisfied. Then, for a fixed $\varepsilon>0$, there exists a unique solution $u_{\varepsilon}\in\VD$ of the penalized problem~\eqref{eq:var_penalized_vi}. Moreover, there exists a positive constant $C>0$ such that
\begin{equation}
 \norm{u_{\varepsilon}}_{\VD}\leq  C\left(\norm{f}_{L^2(\D)}+\norm{\maxx{-\L\psi+f}}_{L^2(\D)}\right).\label{bound_f_penalized}
\end{equation}
\end{lemma}

\begin{proof}
The proof follows the lines, e.g., of \cite[Lemma 2.2]{kinderlehrer80}; we present it here for completeness purposes. First, noting that $\mathcal{H}_{\varepsilon}(u_{\varepsilon}-\psi)\in L^{\infty}(\DD)$, $\varepsilon>0$, we define a nonlinear operator $\LL:\VD\to\VDd$ by
\begin{equation}
 \dual{\LL 
u_{\varepsilon},v}:=a(u_{\varepsilon},v)-\int_{\DD}(\maxx{-\L\psi-f}\mathcal{H}
_{\varepsilon}(u_{\varepsilon}-\psi))v\d\x.\label{nonlinear_oper}
\end{equation}
It is easy to see that $\LL$ is strictly monotone and coercive on $\VD$. Indeed, taking into account that $\mathcal{H}_{\varepsilon}$ is non-increasing, we obtain that
\begin{multline*}
 \dual{\LL u_{\varepsilon}-\LL 
v,u_{\varepsilon}-v}=
a(u_{\varepsilon}-v,u_{\varepsilon}-v)\\
-\int_{\DD}\maxx{-\L\psi-f}(\mathcal{H}_{\varepsilon}
(u_{\varepsilon}-\psi)-\mathcal{H}_{\varepsilon}(v-\psi))(u_{\varepsilon}
-v)\d\x\geq\norm{u_{\varepsilon}-v}_{\V}^2, 
\end{multline*}
which shows that $\LL$ is strictly monotone and coercive. For a fixed $\varepsilon>0$, let there exists a sequence $u_{\varepsilon}^n$ such that $u_{\varepsilon}^n\to u_{\varepsilon}$ strongly in $\VD$. Then, $\LL u^n_{\varepsilon}\rightharpoonup\LL u_{\varepsilon}$ weakly in $\VDd$, which also implies that $\LL$ is continuous on finite-dimensional subspaces of $\VD$; {see~\cite[Definition 1.2]{kinderlehrer80}}. Then, {by existence results for monotone operators~\cite[Chapter~\textrm{III}, Corollary~1.8]{kinderlehrer80}}, and the strict monotonicity of $\LL$, we obtain the existence and uniqueness of $u_{\varepsilon}$ of\eqref{eq:var_penalized_vi} for a fixed $\varepsilon>0$. Then, the bound~\eqref{bound_f_penalized} directly follows from~\eqref{bound_f}, \eqref{poincare}, and~\eqref{eq:f_psi}. 
\end{proof}

Next, under mild assumptions on the data, we state the main regularity result for a nonlocal variational inequality.

\begin{theorem}[Regularity of the Lagrange multiplier]\label{thm:regularity_lagrange}
 Let $f$ and $\psi$ be such that the condition~\eqref{eq:f_psi} holds, and let
$(u,\lambda)\in\VD\times M$ be the unique solution pair of~\eqref{eq:sp} for the kernels $\ker$, satisfying~\eqref{kernel_prop1}  and defined in either {\it Case~1}, {\it Case~2}, or {\it Case~3}. Then, $\L u\in L^2(\D)$, $\lambda\in L^2(\D)$, and $\lambda\leq\maxx{-\L\psi-f}$. 
\end{theorem}

\begin{proof}
First, we verify that $u_{\epsilon}\in\mathcal{K}$. Set $\vt(\x)=u_{\varepsilon}(\x)-\max(u_{\varepsilon}(\x),\psi(\x))=-(\psi(\x)-u_{\varepsilon}(\x))^+\leq 0$. Because {$\psi(\x)\leq 0$ on $\D_I$}, we obtain that $\maxx{\psi-u_{\varepsilon}}\in\VD$ and $\vt\in\VD$. Indeed, for any $v\in\VD$, due to the Lipschitz continuity of the map $v\mapsto{v}^+$, we obtain
\begin{equation*}
 \int_{(\DD)^2}|v^+(\x)-v^+(\xx)|^2\ker(\x,\xx)\d(\x,\xx)
\leq\int_ { (\DD)^2}|v(\x)-v(\xx)|^2\ker(\x,\xx)\d(\x,\xx),
\end{equation*}
i.e., $v^+\in\XD$ and also $v^+\in\VD$. Next, we show that $\vt=0$. From the 
nonlocal Green's identity~\cite{Dunonlocal2012}, we obtain
\begin{equation}
 -\int_{\DD}(\L\psi)\vt\d\x
= 
\intd\intd\left(\psi(\x)-\psi(\xx)\right)\left(\vt(\x)-\vt(\xx)\right)\ker(\x,
\xx)\d\xx\d\x.\label{eq:reg1}
\end{equation}
Subtracting \eqref{eq:var_penalized_vi} and \eqref{eq:reg1}, we obtain
\begin{equation}
a(u_{\varepsilon}-\psi,\vt)=\int_{\DD}\left((-\L\psi-f)^+\mathcal{H}_{
\varepsilon } (u_{\varepsilon}-\psi)+f+\L\psi\right)\vt\d\x.\label{eq:reg3}
\end{equation}
Let $\O_1:=\{\x\in\DD\colon u_{\varepsilon}(\x)\geq\psi(\x)\}$, $\O_2:=\{\x\in\DD\colon u(\x)_{\varepsilon}<\psi(\x)\}$. Then,
\begin{equation}
 \vt(\x)-\vt(\xx)=\begin{cases}
                   0, &\x,\;\xx\in\O_1,\\
                   -(u_{\varepsilon}(\xx)-\psi(\xx)), &\x\in\O_1,\;\xx\in\O_2,\\
u_{\varepsilon}(\x)-\psi(\x), &\x\in\O_2,\;\xx\in\O_1,\\
u_{\varepsilon}(\x)-\psi(\x)-(u_{\varepsilon}(\xx)-\psi(\xx)),&\x,\;\xx\in\O_2.
                  \end{cases}\label{eq:reg2}
\end{equation}
Let $\zz:=u_{\varepsilon}-\psi$; then, $\zz(\x)\geq 0$ on $\O_1$ and $\zz(\x)<0$ on $\O_2$. Taking into account~\eqref{eq:reg2}, we can estimate the left-hand side of~\eqref{eq:reg3} as 
\begin{align*}
 a(\zz,\vt)=&
\intd\intd\left(\zz(\x)-\zz(\xx)\right)\left(\vt(\x)-\vt(\xx)\right)\ker(\x,
\xx)\d\xx\d\x\\
=&\int_{\O_1}\int_{\O_2}(\zz(\x)-\zz(\xx))(-\zz(\xx))\ker(\x,
\xx)\d\xx\d\x\\
&+\int_{\O_2}\int_{\O_1}(\zz(\x)-\zz(\xx))\zz(\x)\ker(\x,
\xx)\d\xx\d\x\\
&+\int_{\O_2}
\int_{\O_2}(\zz(\x)-\zz(\xx))^2\ker(\x,
\xx)\d\xx\d\x\geq 0.
\end{align*}
On the other hand, taking into account the fact that $\vt(\x)=0$ on $\O_1$ and $\mathcal{H}_{\varepsilon}(\zz)=1$ on $\O_2$, the right-hand side of~\eqref{eq:reg3} can be estimated as
\begin{multline*}
 \int_{\DD}\left((-\L\psi-f)^+\mathcal{H}_{
\varepsilon } 
(\zz)+f+\L\psi\right)\vt\d\x\\
=\int_{\O_2}
\left((-\L\psi-f)^++f+\L\psi\right)\vt\d\x\leq 0.
\end{multline*}
Hence, $\vt(\x)=\vt(\xx)=0$ and $u_{\varepsilon}(\x)\geq\psi(\x)$ a.e. on $\DD$.

By Lemma~\ref{lemma:penalized}, $u_{\varepsilon}$ is bounded in $\VD$; then, by the Banach-Alaoglu theorem, there exists a subsequence $u_{\varepsilon^\prime}$ that converges weakly in $\VD$, i.e., $u_{\varepsilon^\prime}\rightharpoonup\widetilde{u}$, $\widetilde{u}\in\VD$. Because $u_{\varepsilon}\in\mathcal{K}$ and $\mathcal{K}$ is closed and convex, we obtain that $\widetilde{u}\in\mathcal{K}$. Next, we show that $\widetilde{u}$ solves~\eqref{eq:vi}. From Lemma~\ref{lemma:penalized}, we know that 
\[\dual{\LLe{\varepsilon^\prime}u_{\varepsilon^\prime},\nu}=\int_{\D}f\nu\d\x, 
\quad\forall\nu\in\VD,\]
where $\varepsilon^{\prime}>0$ and the operator $\LLe{\varepsilon^\prime}$ is defined in~\eqref{nonlinear_oper}. Let $v\in\mathcal{K}$, and suppose {$v\geq\psi+\varepsilon_0d$, where  $d(\x)={\rm dist}(\x,\DI)$, and $\varepsilon_0>0$}. Then, by taking $\nu=u_{\varepsilon^\prime}-v$, and applying Lemma~\ref{lemma:minty}, we obtain 
\begin{equation}
 \dual{\LLe{\varepsilon^\prime} 
v,v-u_{\varepsilon^\prime}}-\int_{\DD}f(v-u_{\varepsilon^\prime})\d\x\geq 0, 
\quad\mbox{for all  } 
v\in\mathcal{K},\; {v\geq\psi+\varepsilon_0d}.\label{reg4}
\end{equation}
{For $\varepsilon<\varepsilon_0d(\x)$, $\mathcal{H}_{\varepsilon}(v(\x)-\psi(\x))=0$; then, by letting $\varepsilon\to 0$, we obtain that $\mathcal{H}_{\varepsilon}(v-\psi)\to 0$ pointwise. Thus,
\begin{equation*}
\int_{\D}(-\L\psi-f)^+\mathcal{H}_{\varepsilon}(v-\psi)(v-u_{\varepsilon})\d\x\to 0,\quad\mbox{as }\varepsilon\to 0,
\end{equation*}
is obtained by additionally using $\mathcal{H}_{\varepsilon} \leq 1$, the uniform boundedness of $u_{\varepsilon}$ in $L^2(\D)$, and~\eqref{eq:f_psi}.} 
Then,~\eqref{reg4} becomes
\begin{equation}
 a(v,v-\widetilde{u})-\int_{\DD}f(v-\widetilde{u})\d\x\geq 0,\quad\mbox{for all 
 } 
v\in\mathcal{K},\; v\geq{\psi+\varepsilon_0d}.\label{reg55}
\end{equation}
{In order to show~\eqref{reg55} for arbitrary $v\in\mathcal{K}$, we define $v_{\varepsilon_0}:=\max(\psi+\varepsilon_0d,v)\geq\psi+\varepsilon_0d$. We note that, since $d(\x)$ is Lipschitz continuous, and $d(\x)=0$ on $\DI$, we have that $v_{\varepsilon_0}\in\V$. In addition, we can verify that $v_{\varepsilon_0}\to v$ strongly in $\V$ as $\varepsilon_0\to 0$. Then, letting $\varepsilon_0\to 0$ and applying Lemma~\ref{lemma:minty} again we arrive at}
\begin{equation*}
 a(\widetilde{u},v-\widetilde{u})-\int_{\DD}f(v-\widetilde{u})\d\x\geq 
0,\quad\mbox{for all 
 } 
v\in\mathcal{K}.
\end{equation*}
Thus, $\widetilde{u}=u$ is the unique solution of~\eqref{eq:vi}, hence, 
$u_{\varepsilon}\rightharpoonup u$ in $\VD$. 

For $\varepsilon>0$ define $\lambda_{\varepsilon}\in\VDd$ 
as $\lambda_{\varepsilon}:=Au_{\varepsilon}-\LL u_{\varepsilon}=Au_{\varepsilon}-f$, where 
$A:\VD\to\VDd$ and is defined in~\eqref{eq:operator_A}. Now letting $\varepsilon\to 
0$, we know that $u_{\varepsilon}\rightharpoonup u$ in $\V$, then 
$\lambda_{\varepsilon}\rightharpoonup Au-f$ weakly in $\VDd$. Now, denoting $\lambda=Au-f$, where 
$\lambda$ is a solution of~\eqref{eq:sp}, we obtain that $\lambda_{\varepsilon}\rightharpoonup\lambda$ in $\VDd$. On the other hand, we know that
\[
\dual{\lambda_{\varepsilon},v}=\int_{\D}\left(\maxx{-\L\psi-f}\mathcal{H}_{
\varepsilon}(u_{\varepsilon}-\psi)\right)v\d\x\quad\forall v\in\VD,
\]
{and hence $\lambda_{\varepsilon}$ can be identified with an element in $L^2(\D)$. Furthermore, for some constant $C>0$, independent of $\varepsilon$, $\norm{\lambda_{\varepsilon}}_{L^2(\D)}\leq C$. In fact, the following holds pointwise almost everywhere in $L^2(\D)$:
\begin{equation}\label{eq:dual_stimate}
\lambda_{\varepsilon}=\maxx{-\L\psi-f}\mathcal{H}_{
\varepsilon}(u_{\varepsilon}-\psi)\leq \maxx{-\L\psi-f}. 
\end{equation}
Hence, there exists a subsequence $\lambda_{\varepsilon^\prime}\rightharpoonup\lambda$ in $L^2(\D)$. Because $(u,\lambda)\in\VD\times\VDd$ is the unique solution pair of~\eqref{eq:sp}, it follows that $\lambda_{\varepsilon}\rightharpoonup\lambda$ in $L^2(\D)$ and the weak limits in $\VDd$ and $L^2(\D)$ coincide, thus  $\lambda\in L^2(\D)$, and, additionally, the pointwise relation~\eqref{eq:dual_stimate} is preserved in the limit. This concludes the proof}.
\end{proof}

\begin{corollary}[Improved regularity of the 
solution]\label{corollary:reg_vi_primal}
Let $\D$ be a domain with $C^{\infty}$ boundary $\partial\D$, and let $f$ and $\psi$ be such that the condition~\eqref{eq:f_psi} holds, and with the kernel $\ker$ defined in~\eqref{kernel_case1} for $\delta>0$ and $\sigma(\x,\xx)=\sigma$. Then, for the solution pair $(u,\lambda)\in\VD\times M$ of the variational inequality~\eqref{eq:sp} we obtain that $\L u\in L^2(\D)$, $\lambda\in L^2(\D)$ and $u\in H_\D^{s+\alpha}(\R^n)$, where $\alpha=\min\{s,1/2-\varepsilon\}$ for any $\varepsilon>0$.
\end{corollary}

\begin{proof}
From the Theorem~\ref{thm:regularity_lagrange} we obtain that $\lambda\in L^2(\D)$, and then we apply the regularity results~\eqref{nonl_regularity} for $u$ as a solution of a linear problem with the right hand-side $\lambda+f\in L^2(\D)$.
\end{proof}

\subsection{Discretization of the nonlocal variational 
inequality}\label{subsec:disc_vi}
We subdivide $\D\cup\DI$ by a triangulation $\{\mathcal{T}_\N\}_\N$ in the same manner as described in Section~\ref{subsec:disc_linear}. Similarly, the primal space $V$ is approximated by $\V_\N$ defined as in~\eqref{eq:V_discrete}. 

For the discrete Lagrange multiplier space $W_\N$, we use discontinuous piecewise linear
biorthogonal basis functions, defined with respect to the same mesh as the basis functions of $V_\N$; see~\cite{Woh00a}. Following the same notation as in Section~\ref{subsec:disc_linear}, we define $W_\N=\spann\{\xi_q,\; q\in\mathcal{J}_\N^m\}$, where $\xi_p$, $p\in\mathcal{J}_\N^l$, satisfy a local biorthogonality relation
\begin{equation*}
 \int_{K}\xi_q\phi_p=\delta_{p,q}\int_{K}\phi_{p}\geq 0.
\end{equation*}
Then, the discrete Lagrange multiplier cone $M_\N\subset W_\N$ is defined as 
\begin{equation}
 M_{\N}=\spann_+\{\chi_q\}_{q\in\mathcal{J}_\N^m}:=\Big\{\eta\in W_\N\colon
\eta=\sum_{q\in\mathcal{J}_\N^m} {\alpha}_q\chi_q ,
\; \alpha_q\geq 0\Big\}.
\end{equation}
For these settings, we can guarantee the inf-sup stability of $b(\cdot,\cdot)$ on the pair of discrete spaces $W_\N\times V_\N$. We remark that $\xi_q\not\in M$, i.e., $M_\N\not\subset M$, so that we obtain a non-conforming approximation with respect to the Lagrange multiplier.

The approximation of~\eqref{eq:sp} by the discrete saddle-point problem reads as follows: find $u_\N\in V_\N$ and $\lambda_\N\in M_\N$ such that
\begin{subequations}\begin{align}
a(u_\N,v_\N)-b(\lambda_\N,v_\N)&=\dual{f,v_\N} \quad&&\forall v_\N\in V_\N,\\
b(\eta_\N-\lambda_\N,u_\N)&\geq b(\eta_\N-\lambda_\N,\psi) 
\quad&&\forall\eta_\N\in 
M_\N.
\end{align}\label{eq:spy-disc}\end{subequations}
Because the discrete inf-sup stability of $b(\cdot,\cdot)$ holds on $W_\N\times\V_\N$, we also 
have the existence and uniqueness of the solution of the discrete problem~\eqref{eq:spy-disc}.
The problem~\eqref{eq:spy-disc} can be solved by a semi-smooth Newton method, and, in particular, we can employ the primal-dual active set strategy~\cite{kunisch} that is known to be locally superlinearly convergent.


\section{Numerical results}\label{sec:numerics}
In this section, we present numerical results for the nonlocal variational equality and inequality problems that illustrate our theoretical findings and provide an outlook on  open questions and possible future directions.

\subsection{Model settings}
We consider one-dimensional problems, where the the computational domain $\D$ is set to $\D=(0,1)$ and is discretized with the uniform mesh of a mesh size $h=1/N$, $N\in\mathbb{N}$. Unless otherwise stated, we set $h=2^{-9}$. The interaction domain is then given by $\DI=(-\delta,0)\cup(1,1+\delta)$.

Due to the nonlocality of the problems, the corresponding matrices, in general, are not sparse (in contrast to the finite element matrices for a local problem). This also increases the computational cost of assembling the matrices, and eventually of solving the discrete problems. {In certain cases, e.g., when $\D$ is a $n-$dimensional hyper-rectangle and the kernel is translation invariant, the matrix possesses a multilevel Toeplitz structure for regular grids. In this case, the assembly time can be significantly reduced by only computing the first row of the matrix; see~\cite{vollmann2017}.}

\subsubsection*{Comparison to the local variational inequality}

First, we consider solutions of the variational inequality~\eqref{eq:spy-disc} with $\ker$ defined as in {\it Case 1} with $\sigma=(2-2s)/\delta^{2-2s}$ and an obstacle functional $\psi$ defined as
\begin{equation}\label{eq:psi_smooth}
 \psi=\max(-3(x-0.5)^2+0.25,0).
\end{equation}
Snapshots of solutions for different values of $\delta>0$ are presented in Figure \ref{fig:sol_delta}. We also plot the solution of the local obstacle problem, i.e., for $-\L=-\Delta$. As we can see, as $\delta\to 0$, the nonlocal solutions $u_h$ and $\lambda_h$ converge to the solution of the local variational inequality, as is expected. {We also note that $\psi\in H^{2s}(\D)$, $s\in(0,1)$ and it can be easily verified that the condition~\eqref{eq:f_psi} is fulfilled. Hence, $\lambda\in L^2(\D)$ according to Theorem~\ref{thm:regularity_lagrange}, which is also confirmed by the numerical examples.}

In addition, we can make another interesting observation. It is well known that for local obstacle problems, the Lagrange multiplier has jumps, as is also clearly seen from Figure~\ref{fig:sol_delta}. In contrast, for the nonlocal problem, we observe some ``smoothing'' effect with respect to increasing the interaction radius $\delta$. 

In Figure~\ref{fig:sol_delta_s}, solutions of~\eqref{eq:spy-disc} for different values of $s\in(0,1)$ with $\delta=2$ and $\ker$ defined as in {\it Case 1} with $\sigma=1/2\delta^2$. As is expected, for $s\to 1$, the solution of the nonlocal variational inequality converges to the solution of the local problem. Similarly, as in the previous example, we also observe a ``smoothing'' effect in the Lagrange multipliers for decreasing $s$. Our theory provides $L^2(\D)$ regularity of the Lagrange multipliers; however from both examples, the numerical results suggest a possible higher regularity of the nonlocal Lagrange multipliers in contrast to the local ones. This question is left for future investigation. 

\begin{figure}[ht]
\centering
\includegraphics[width=0.48\textwidth, height=4cm]{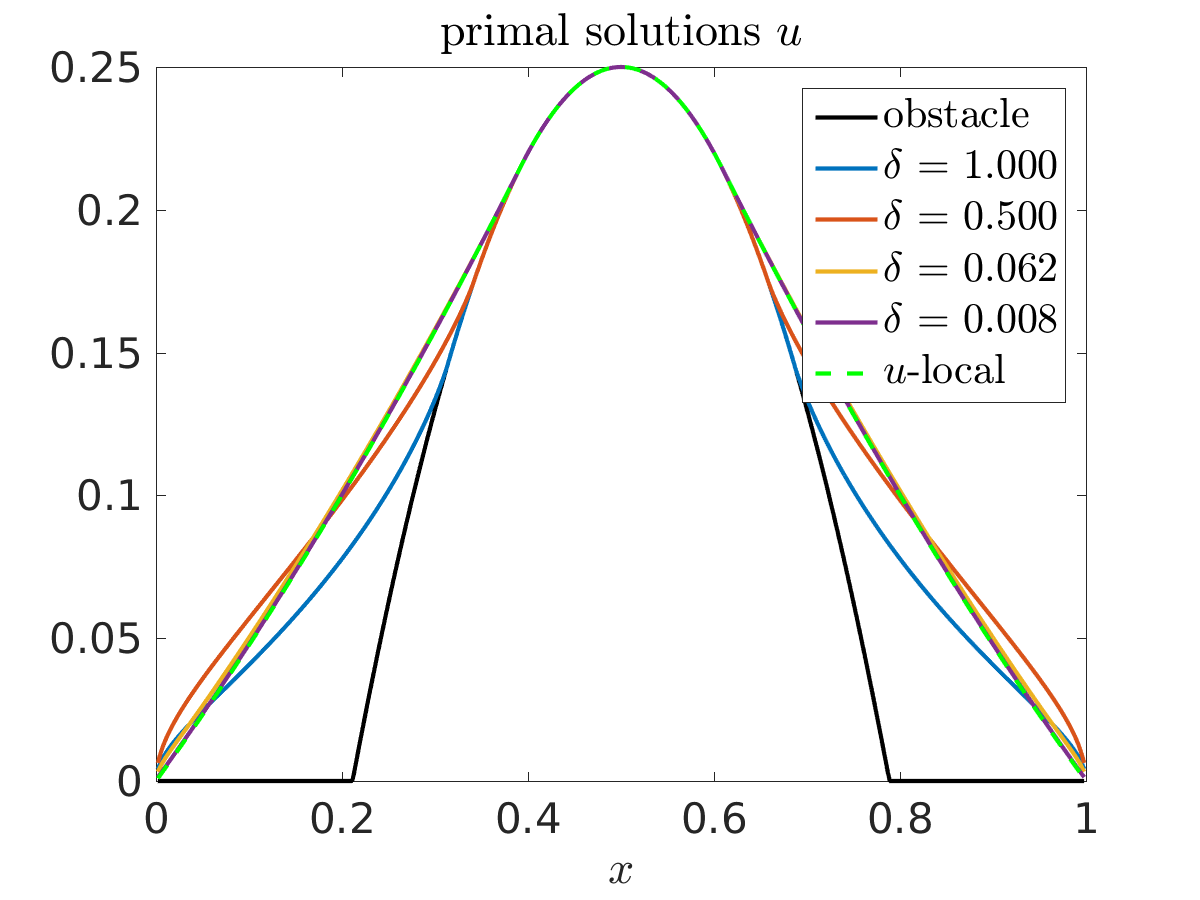}
\includegraphics[width=0.48\textwidth, height=4cm]{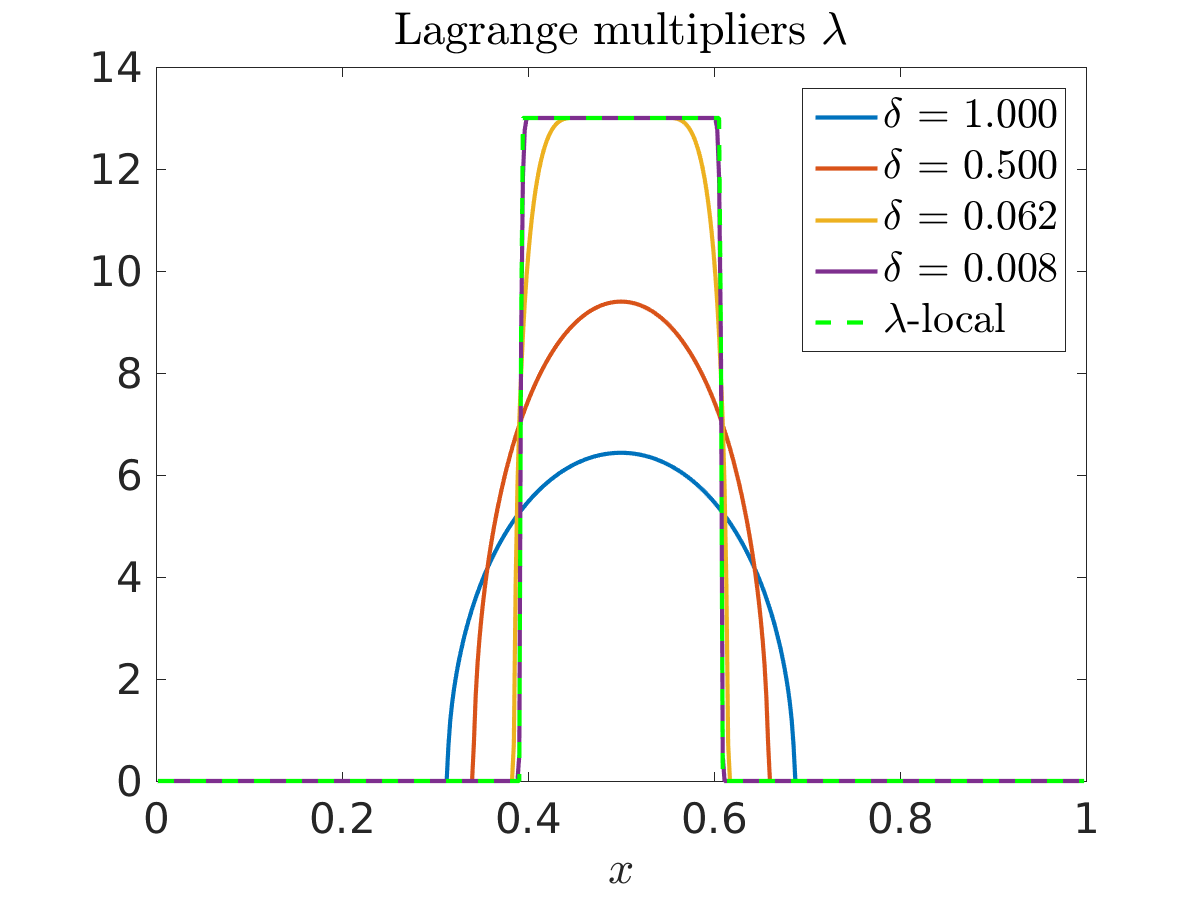}
\caption{Solutions $(u_\N,\lambda_\N)$ for different interaction radii $\delta$ for $s=0.5$, $\sigma=(1-s)/\delta^{2-2s}$, $f=-1,$ and $\psi$ defined in~\eqref{eq:psi_smooth} and also for the corresponding local problem.}\label{fig:sol_delta}
\end{figure}
\begin{figure}[ht]
\centering
\includegraphics[width=0.48\textwidth, height=4cm]
{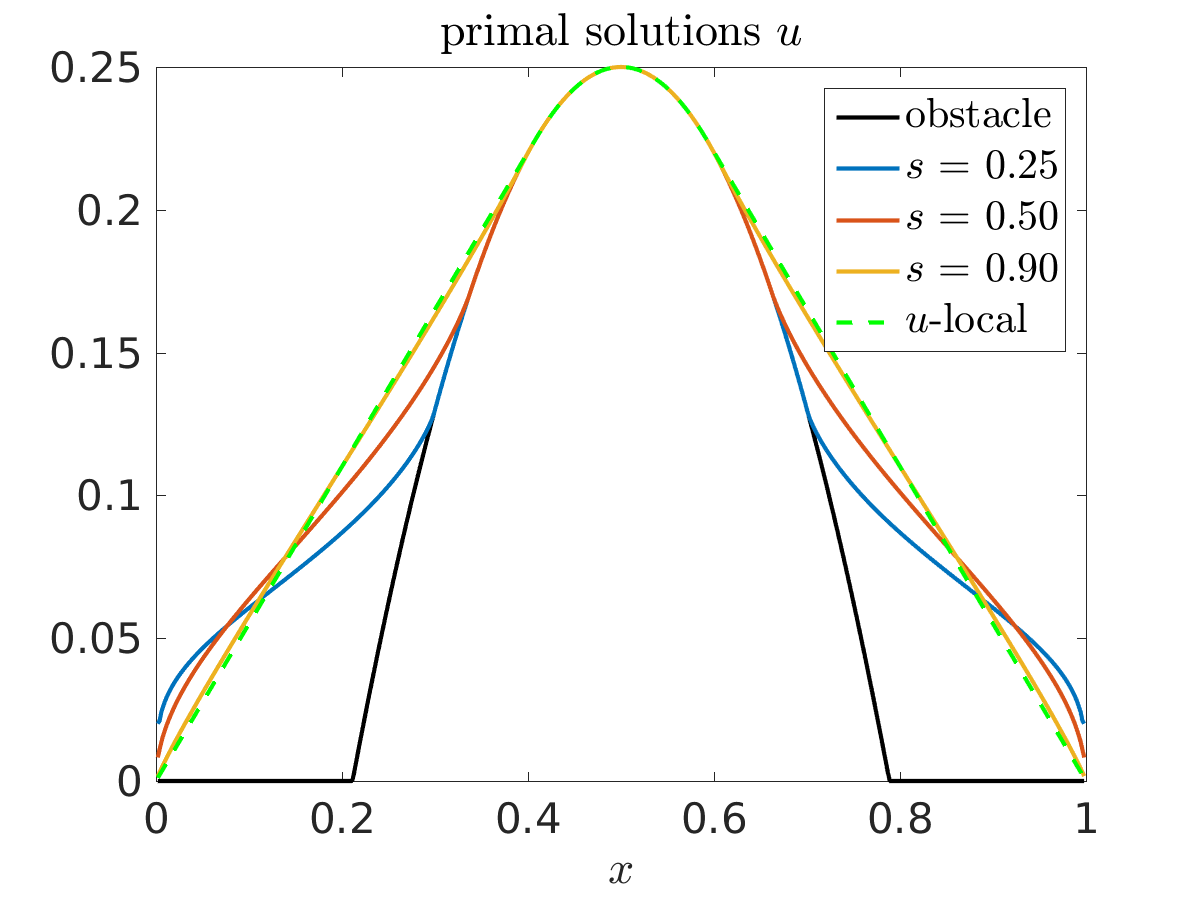}
\includegraphics[width=0.48\textwidth, height=4cm]
{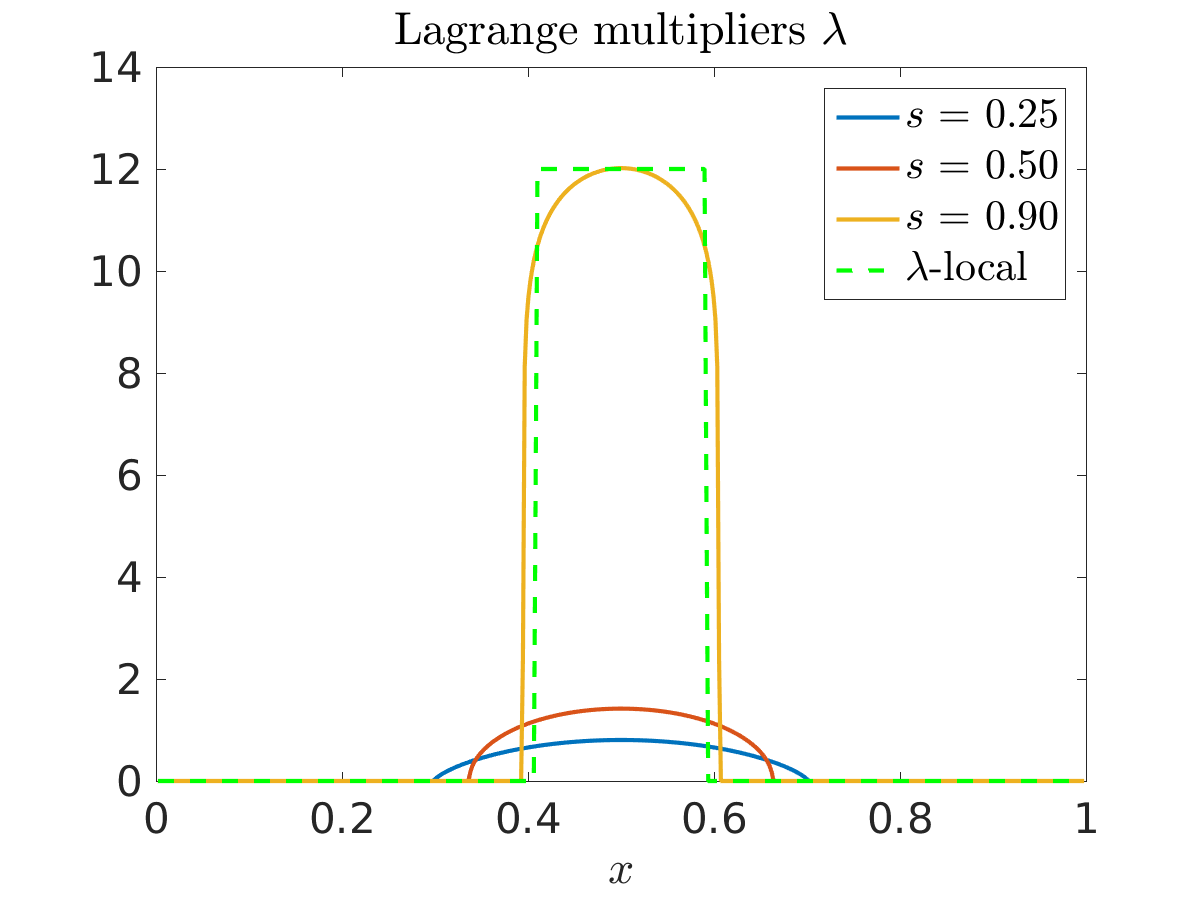}
\caption{Solutions $(u_\N,\lambda_\N)$ for different values of $s\in(0,1)$ and $f=0$, $\sigma=1/2\delta^2$, $\delta=2$, and $\psi$ defined in~\eqref{eq:psi_smooth} and also for the corresponding local problem.}\label{fig:sol_delta_s}
\end{figure}

Next, we consider a less regular obstacle functional
\begin{equation}\label{eq:psi_nonsmooth}
 \psi=
\begin{cases}
 0.02, &\mbox{ for } x\in[{1}/{6},{2}/{6}],\\
 0.24(x-{2}/{3}), &\mbox{ for } x\in[ { 2 }/{ 3 }
,{3}/{4} ],\\
0.24({5}/{6}-x),&\mbox{ for } x\in[ {3}/{4} , {5}/{6} ].
\end{cases}
\end{equation}
For $\psi$ defined in~\eqref{eq:psi_nonsmooth}, $\ker$ defined as in {\it Case 1} with $\sigma=1/2\delta^2$, and $\delta=2$, we plot the solution for different values of $s\in(0,1)$ in Figure~\ref{fig:sol_delta_s_nonsmooth}. In the same figure, we also plot the solution of the local problem.
\begin{figure}[ht!]
\centering
\includegraphics[width=0.48\textwidth, height=4cm]
{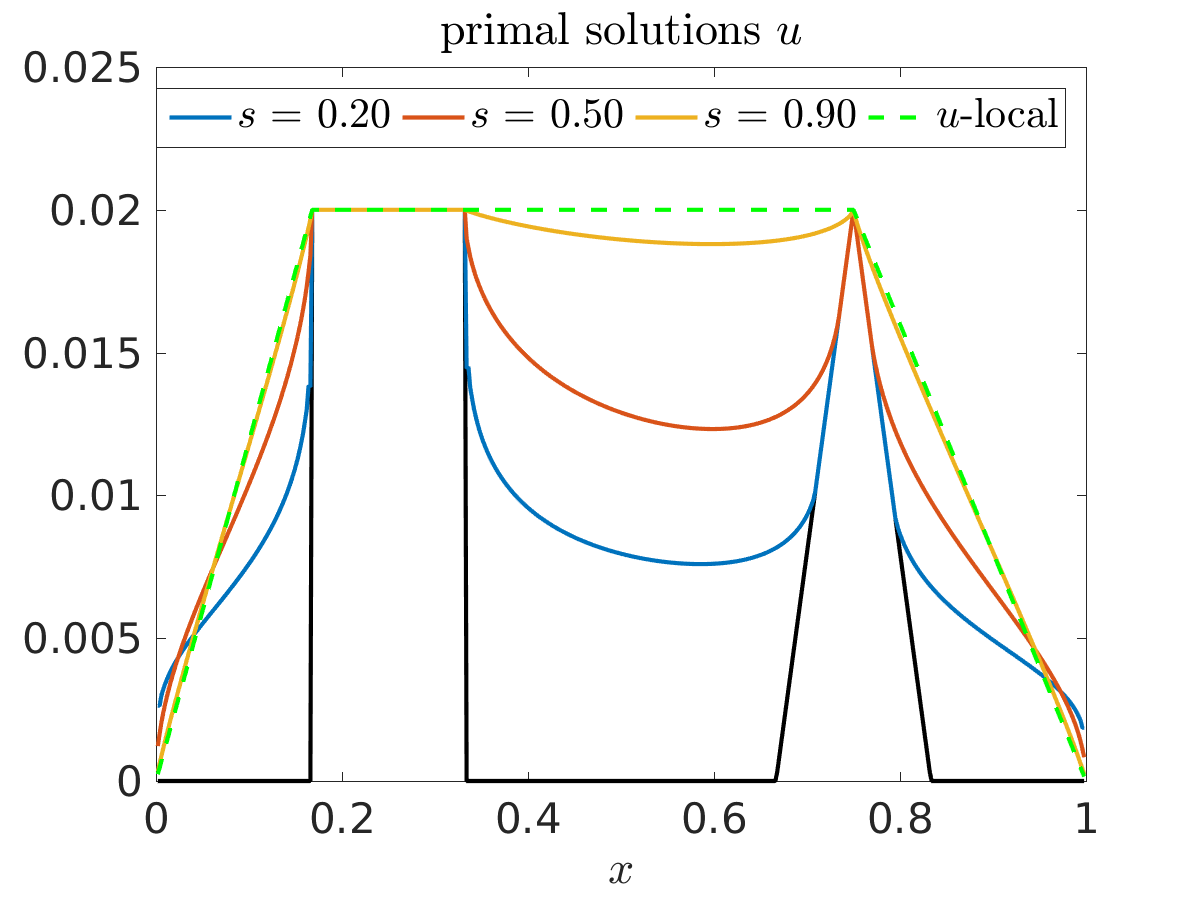}
\includegraphics[width=0.48\textwidth, height=4cm]
{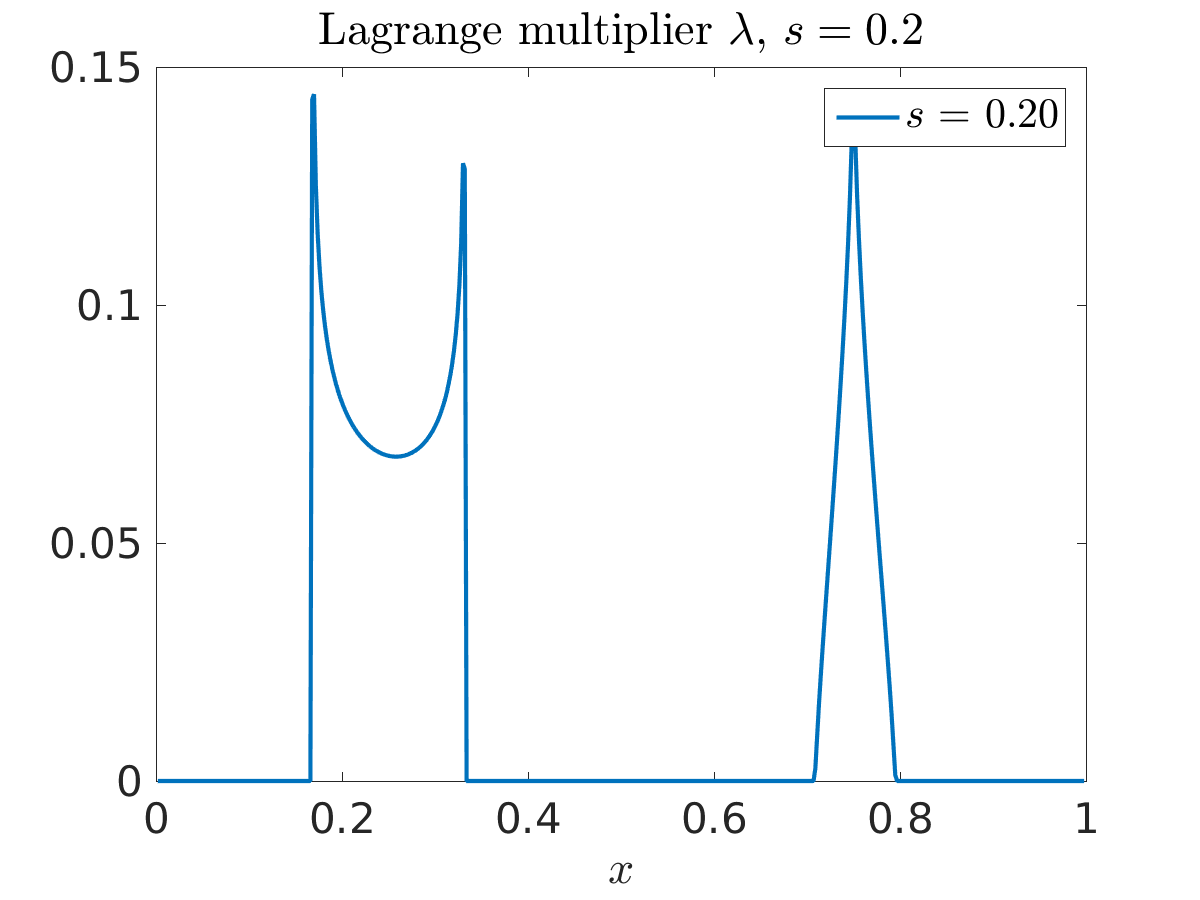}\\
\includegraphics[width=0.48\textwidth, height=4cm]
{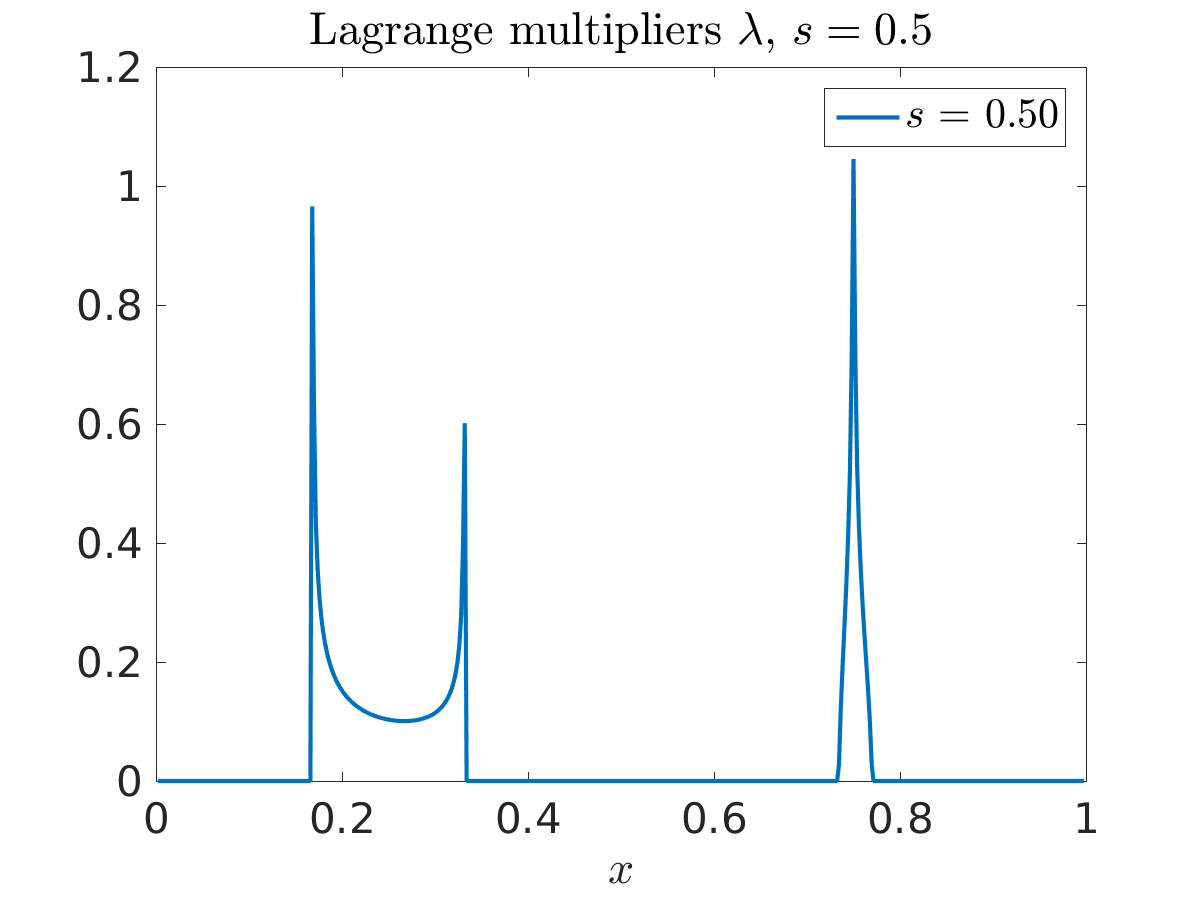}
\includegraphics[width=0.48\textwidth, height=4cm]
{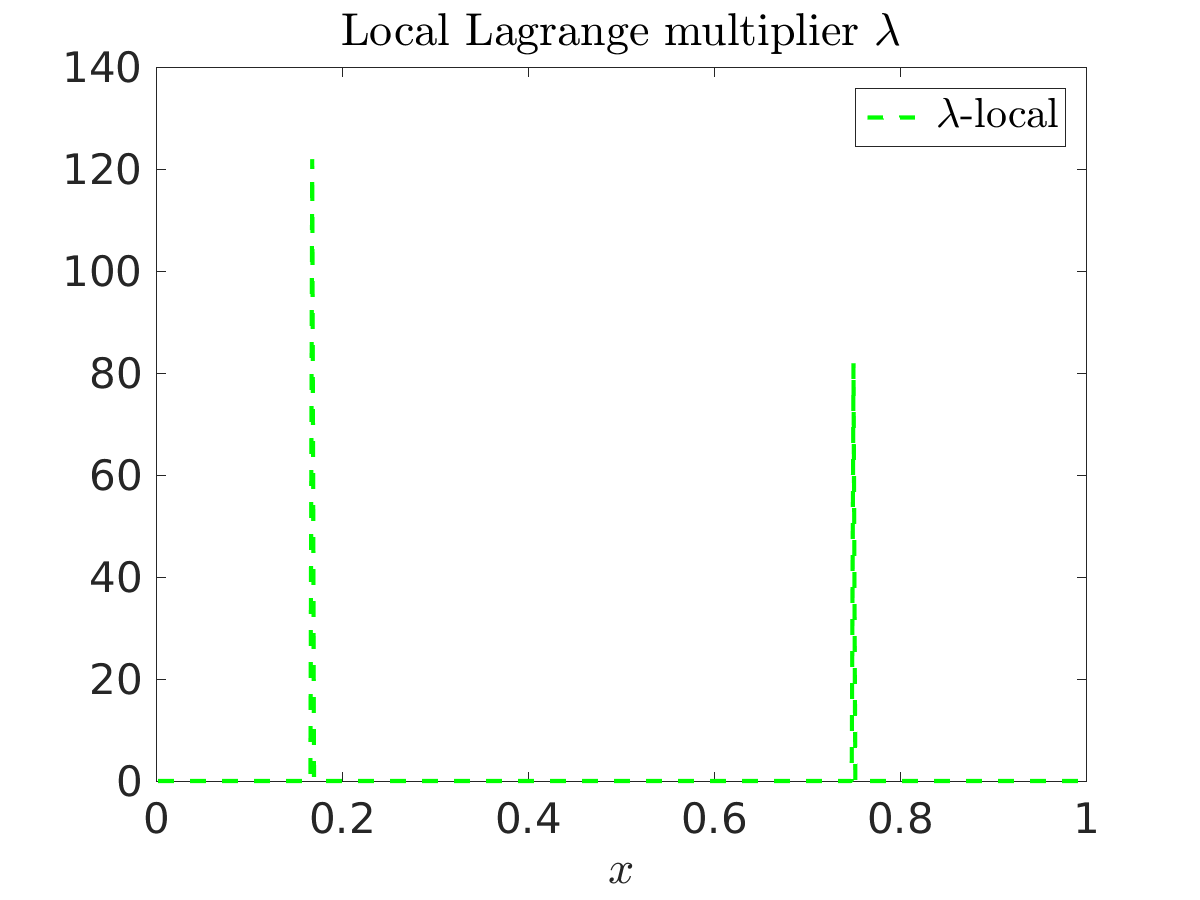}
\caption{The solution $(u_\N,\lambda_\N)$ for 
different values of $s\in(0,1)$ and $f=0$, $\sigma=1/2\delta^2$, $\delta=2$, 
and $\psi$ defined in~\eqref{eq:psi_nonsmooth}.} 
\label{fig:sol_delta_s_nonsmooth}
\end{figure}

{In contrast to the previous example, in this case $\psi\not\in H^{1}$, i.e., $\psi\not\in H^{2s}(\D)$ for all $s\in(0,1)$. While $\psi|_{[1/2,1]}$ belongs to $H^1(\D)$, the remaining part $\psi|_{[0,1/2)}$ is only in $H^{1/2-\varepsilon}(\D)$. Correspondingly, the Lagrange multiplier of the local problem, which is a Dirac delta function, is not in $L^2(\D)$. For the nonlocal problem, the situation is slightly different. Namely, for $s<1/4-\varepsilon$, $\varepsilon>0$, $\psi\in H^{2s}(\D)$, and it can be verified that the condition~\eqref{eq:f_psi} holds, and thus the Lagrange multiplier of the nonlocal variational inequality is in $L^2(\D)$. Correspondingly, for the numerical example for $s=0.2$ in Figure \ref{fig:sol_delta_s_nonsmooth}, we observe the $L^2$ regularity of the nonlocal Lagrange multiplier. Moreover, from Corollary~\ref{corollary:reg_vi_primal}, we know that in this case $u\in H^{2s}_\D(\R^n)$, i.e., $u\in H^{0.4}_\D(\R^n)$. In addition, it is known that the nonlocal solution can admit jump discontinuities for $s<1/2$, and in the present case we can also clearly observe a discontinuity in the solution, which is depicted on the right-top plot in Figure~\ref{fig:sol_delta_s_nonsmooth}.

In contrast, for $s=0.5$, the right part of $\psi$, i.e.,  $\psi|_{[1/2,1]}$, still belongs to $H^{2s}(\D)$, however  left part $\psi|_{[0,1/2)}$ does not. Therefore, on the bottom-left picture for $s=0.5$ in Figure~\ref{fig:sol_delta_s_nonsmooth}, we can only say that the right part of the Lagrange multiplier is in $L^2(\D)$. 

\subsubsection*{Comparison to the variational inequality with the fractional Laplacian}
We study the convergence of the solution of the nonlocal variational inequality~\eqref{eq:spy-disc} to the solution of the variational inequality with the fractional Laplace operator~\eqref{eq:sp_fl}. Because the analytic solution of the variational inequality is in general not known, we compare the corresponding finite element surrogates. In Table~\ref{table:delta_convergence}, we report the approximation errors between the nonlocal solution $u^{\delta}_\N$ of~\eqref{eq:spy-disc} with $\ker$ defined in~\eqref{kernel_case1} with $\sigma(\x,\xx)={c_{n,s}/2}$, and the finite element approximation $u_\N^{FL}$ of~\eqref{eq:sp_fl} for different $\delta>0$ and fixed $s=0.5$. We observe the rate one for both, the energy and $L^2$ errors. This illustrates the theoretical findings in Proposition~\ref{prop:LP_nonlocal}, i.e., the error is proportional to $O(\delta^{-2s})$, $s\in(0,1)$. 

\begin{table}[ht!]
 \centering
 \begin{tabular}{lcc|cc}
\toprule
$\delta$ & energy error & rate in $\V$-norm & $L^2$-error& rate in $L^2$-norm\\
\midrule
$2^{3}$& \rr{9.7}{3}&--& \rr{1.6}{3} &-- \\
$2^{4}$& \rr{4.8}{3}&1.01& \rr{8.0}{4} &1.02 \\
$2^{5}$& \rr{2.4}{3}&1.00& \rr{3.9}{4} &1.00 \\
$2^{6}$& \rr{1.2}{3}&1.00& \rr{2.0}{4} &1.00 \\
$2^{7}$& \rr{6.0}{4}&1.00&  \rr{9.8}{4} &1.00 \\
$2^{8}$& \rr{2.9}{4}&1.00& \rr{4.9}{5} &1.00 \\
\bottomrule
\end{tabular}
\caption{Convergence of the energy error $\norm{u^{\delta}_{\N}-{u}_\N^{FL}}_\V$ and $L^2$-error $\norm{u^{\delta}_{\N}-{u}_\N^{FL}}_{L^2(\DD)}$ with respect to the interaction radii $\delta$ and fixed $s=0.5$, and the corresponding convergence rates.}\label{table:delta_convergence}
\end{table}

\subsubsection*{Convergence of the finite element approximation for a linear 
problem}
Now, we study the convergence of the finite element approximation for the linear problem~\eqref{nonl_linear_var} with respect to the mesh size $h$. We consider the error between the finite element solution $u_{\N}$ of~\eqref{nonl_linear_disc} on different grid levels and the surrogate $\bar{u}$, which is also a finite element solution computed on a fine mesh of a grid size $h=2^{-12}$. Here, we consider $\ker$ defined in~\eqref{kernel_case1} with $\sigma=1$. 

In Table~\ref{table:h_lin_conv} we present convergence results for $f=1$, $s=0.5$, and $\delta=0.5$. Because for these settings the solution admits $H^{1-\varepsilon}(\DD)$ regularity, $\varepsilon>0$, we observe an $O(h^{0.5})$ rate of convergence in the energy norm and $O(h)$ in the $L^2$ norm, which illustrates the theoretical results in Proposition~\ref{prop:convergence_rates}.

\begin{table}[ht!]
 \centering
 \begin{tabular}{lcc|cc}
\toprule
$h$ & energy error & rate in $\V$-norm & $L^2$-error& rate in $L^2$-norm\\
\midrule
$2^{-3}$& \rr{9.27}{2}&--& \rr{1.00}{2} &-- \\
$2^{-4}$& \rr{6.59}{2}&0.492& \rr{5.21}{3} &0.950 \\
$2^{-5}$& \rr{4.66}{2}&0.498& \rr{2.67}{3} &0.963 \\
$2^{-6}$& \rr{3.27}{2}&0.513& \rr{1.34}{3} &0.992 \\
$2^{-7}$& \rr{2.24}{2}&0.544& \rr{6.54}{4} &1.041 \\
\bottomrule
 \end{tabular}
\caption{{\it Linear problem}: Convergence of the energy error $\norm{u_{\N}-\bar{u}}_\V$ and $L^2$-error $\norm{u_{\N}-\bar{u}}_{L^2(\DD)}$ with respect to the grid size $h$, and the corresponding convergence rates, with $s=0.5$, $\delta=0.5$, $f=1$.}\label{table:h_lin_conv}
\end{table}

To study the convergence behavior of the finite element approximation with respect to different $s\in(0,1)$, we report the corresponding rates of convergence in Table~\ref{table:h_lin_conv_s}. Here, as previously, the results illustrates the theoretical ones~\eqref{eq:convergence_rates}. 
\begin{table}[ht!]
\centering
\begin{tabular}{c|cccc|cccc}
 \toprule
&\multicolumn{3}{c}{energy 
error}&\multicolumn{5}{c}{$L^2$-error}\\
\backslashbox[1cm][c]{$h$}{$s$}&$0.1$ &$0.25$ &$0.5$&$0.75$ 
&$0.1$ & $0.25$ & $0.5$&$0.75$\\
 \midrule
$2^{-5}$&0.48  &0.49 &0.50 &0.55   &0.68 &0.79 &0.97 &1.15\\
$2^{-6}$&0.49  &0.49 &0.50 &0.53   &0.67 &0.79 &0.97 &1.11\\
$2^{-7}$&0.50  &0.50 &0.51 &0.52   &0.66 &0.79 &0.98 &1.20\\
$2^{-8}$&0.51  &0.51 &0.52 &0.53   &0.66 &0.80 &1.00 &1.11\\
$2^{-9}$&0.54  &0.54 &0.54 &0.55   &0.68 &0.82 &1.04 &1.15\\
\bottomrule
\end{tabular}
\caption{{\it Linear problem}: Rates of convergence of the error in the energy norm $\norm{u_{\N}-\bar{u}}_\V$ and $L^2$-norm $\norm{u_{\N}-\bar{u}}_{L^2(\DD)}$ with respect to the grid size $h$ for different $s$ and fixed $\delta=1$. The reference solution is computed on the fine mesh with $h=2^{-12}$. }\label{table:h_lin_conv_s}
\end{table}

\subsubsection*{Convergence of the finite element approximation for the variational 
inequality}
We also report on the convergence of the finite element method for the nonlocal 
variational inequality~\eqref{eq:sp}. Here, as before, the surrogate solution 
$\bar{u}$, which is a finite element solution, computed 
on a fine mesh of
grid size $h=2^{-12}$, and the kernel $\ker$ defined 
in~\eqref{kernel_case1} with 
$\sigma=1$. For $f=0$, $s=0.5$, $\delta=1$, and an obstacle functional $\psi$ 
defined in~\eqref{eq:psi_smooth}, the convergence of the energy and $L^2$ 
errors together with the corresponding rates are presented in 
Table~\ref{table:h_vi_conv}. For other values of $s\in(0,1)$ we also report 
the convergence rates in Table~\ref{table:h_vi_conv_s}. As 
we can observe, the numerical results suggest the same convergence order for 
$u_\N$ as in the case of the linear problem.
\begin{table}[ht!]
 \centering
 \begin{tabular}{lcc|cc}
\toprule
$h$ & energy error & rate in $\V$-norm & $L^2$-error& rate in $L^2$-norm\\
\midrule
$2^{-5}$& \rr{3.61}{2}&0.556& \rr{1.62}{3} &1.200 \\
$2^{-6}$& \rr{2.54}{2}&0.508& \rr{8.22}{4} &0.985\\
$2^{-7}$& \rr{1.78}{2}&0.508& \rr{4.18}{4} &0.976 \\
$2^{-8}$& \rr{1.24}{2}&0.520& \rr{2.10}{4} &0.987 \\
$2^{-9}$& \rr{8.50}{3}&0.548& \rr{1.02}{4} &1.041 \\
\bottomrule
 \end{tabular}
\caption{{\it Variational inequality}: Convergence of the energy error 
$\norm{u_{\N}-\bar{u}}_\V$ and 
$L^2$-error $\norm{u_{\N}-\bar{u}}_{L^2(\DD)}$ with respect to the grid 
size $h$, and the corresponding convergence rates with $s=0.5$, $\delta=1$, 
$f=0$, $\sigma=1$, and $\psi$ defined in~\eqref{eq:psi_smooth}. The reference 
solution is computed on the fine mesh with $h=2^{-12}$.}\label{table:h_vi_conv}
\end{table}

\begin{table}[ht!]
\centering
\begin{tabular}{c|cccc|cccc}
 \toprule
&\multicolumn{3}{c}{energy 
error}&\multicolumn{5}{c}{$L^2$-error}\\
\backslashbox[1cm][c]{$h$}{$s$}&$0.1$ &$0.25$ &$0.5$&$0.75$ 
&$0.1$ & $0.25$ & $0.5$&$0.75$\\
 \midrule
$2^{-5}$&0.52&    0.49&    0.56&    0.70&    0.69&    0.80&    1.20&   1.40\\
$2^{-6}$&0.50&    0.49&    0.51&    0.58&    0.67&    0.77&    0.98&   1.23\\
$2^{-7}$&0.50&    0.50&    0.51&    0.53&    0.65&    0.77&    0.97&  1.12\\
$2^{-8}$&0.51&    0.51&    0.52&    0.53&    0.65&    0.78&    0.98&  1.15\\
$2^{-9}$&0.54&    0.54&    0.55&    0.55&    0.67&    0.81&    1.04&  1.18\\
\bottomrule
\end{tabular}
\caption{{\it Variational inequality}: Rates of convergence of the error in the energy norm $\norm{u_{\N}-\bar{u}}_\V$ and $L^2$-norm $\norm{u_{\N}-\bar{u}}_{L^2(\DD)}$ with respect to the grid size $h$ for different $s$ and fixed $\delta=1$. The reference solution is computed on the fine mesh with $h=2^{-12}$. 
}\label{table:h_vi_conv_s}
\end{table}

\section{Concluding remarks}
In this paper, we provide regularity results for nonlocal linear and obstacle problems. For both cases, we have extended the regularity results~\cite{grubb2015} to the truncated fractional Laplace kernels. For a general class of kernels, we derive an improved regularity estimate for the Lagrange multiplier. We present the discretization of both problems by a finite element method and prove a priori error estimates for the linear problem. The development of the a priori error estimates for nonlocal obstacle problem is beyond the scope of this paper and is a subject for future investigation.


\bibliographystyle{amsplain}
\bibliography{references}

\providecommand{\bysame}{\leavevmode\hbox to3em{\hrulefill}\thinspace}
\providecommand{\MR}{\relax\ifhmode\unskip\space\fi MR }
\providecommand{\MRhref}[2]{%
  \href{http://www.ams.org/mathscinet-getitem?mr=#1}{#2}
}
\providecommand{\href}[2]{#2}
\begin{thebibliography}{10}

\bibitem{acosta2017}
G.~Acosta and J.P. Borthagaray, \emph{A fractional {L}aplace equation:
  regularity of solutions and finite element approximations}, SIAM J. Numer.
  Anal. \textbf{55} (2017), no.~2, 472--495. \MR{3620141}

\bibitem{bensoussan}
A.~Bensoussan and J.-L. Lions, \emph{Applications of variational inequalities
  in stochastic control}, Studies in Mathematics and its Applications, vol.~12,
  North-Holland Publishing Co., Amsterdam-New York, 1982, Translated from the
  French. \MR{653144 (83e:49012)}

\bibitem{bonito2017}
A.~{Bonito}, J.~P. {Borthagaray}, R.~H. {Nochetto}, E.~{Otarola}, and A.~J.
  {Salgado}, \emph{{Numerical Methods for Fractional Diffusion}}, ArXiv
  e-prints (2017).

\bibitem{brennerscott}
S.~C. Brenner and L.~R. Scott, \emph{The mathematical theory of finite element
  methods}, third ed., Texts in Applied Mathematics, vol.~15, Springer, New
  York, 2008. \MR{2373954}

\bibitem{buades2010}
A.~Buades, B.~Coll, and J.~M. Morel, \emph{Image denoising methods. {A} new
  nonlocal principle}, SIAM Rev. \textbf{52} (2010), no.~1, 113--147, Reprint
  of ``A review of image denoising algorithms, with a new one'' [MR2162865].
  \MR{2608636}

\bibitem{caffarelli2017}
L.~Caffarelli, X.~Ros-Oton, and J.~Serra, \emph{Obstacle problems for
  integro-differential operators: regularity of solutions and free boundaries},
  Invent. Math. \textbf{208} (2017), no.~3, 1155--1211. \MR{3648978}

\bibitem{caffarelli2007}
L.~Caffarelli and L.~Silvestre, \emph{An extension problem related to the
  fractional {L}aplacian}, Comm. Partial Differential Equations \textbf{32}
  (2007), no.~7-9, 1245--1260. \MR{2354493}

\bibitem{ciarlet}
P.~G. Ciarlet, \emph{The finite element method for elliptic problems},
  North-Holland Publishing Co., Amsterdam-New York-Oxford, 1978, Studies in
  Mathematics and its Applications, Vol. 4. \MR{0520174}

\bibitem{cont2004}
R.~Cont and P.~Tankov, \emph{Financial modelling with jump processes}, Chapman
  \& Hall/CRC Financial Mathematics Series, Chapman \& Hall/CRC, Boca Raton,
  FL, 2004. \MR{2042661}

\bibitem{conv_diff2017}
M.~D'Elia, Q.~Du, M.~Gunzburger, and R.~Lehoucq, \emph{Nonlocal
  convection-diffusion problems on bounded domains and finite-range jump
  processes}, Comput. Methods Appl. Math. \textbf{17} (2017), no.~4, 707--722.
  \MR{3709057}

\bibitem{DELIAlaplacian}
M.~D'Elia and M.~Gunzburger, \emph{The fractional {L}aplacian operator on
  bounded domains as a special case of the nonlocal diffusion operator},
  Comput. Math. Appl. \textbf{66} (2013), no.~7, 1245--1260. \MR{3096457}

\bibitem{Dunonlocal2012}
Q.~Du, M.~Gunzburger, R.~B. Lehoucq, and K.~Zhou, \emph{Analysis and
  approximation of nonlocal diffusion problems with volume constraints}, SIAM
  Rev. \textbf{54} (2012), no.~4, 667--696. \MR{3023366}

\bibitem{Dunonlocal2013}
\bysame, \emph{A non-local vector calculus, non-local volume-constrained
  problems, and non-local balance laws}, Mathematical Models and Methods in
  Applied Sciences \textbf{23} (2013), no.~03, 493--540.

\bibitem{gilboa2007}
G.~Gilboa and S.~Osher, \emph{Nonlocal linear image regularization and
  supervised segmentation}, Multiscale Model. Simul. \textbf{6} (2007), no.~2,
  595--630. \MR{2338496}

\bibitem{glowinski}
R.~Glowinski, \emph{Numerical methods for nonlinear variational problems},
  Scientific Computation, Springer, 2008.

\bibitem{grisvard}
P.~Grisvard, \emph{Elliptic problems in nonsmooth domains}, Monographs and
  Studies in Mathematics, vol.~24, Pitman (Advanced Publishing Program),
  Boston, MA, 1985. \MR{775683}

\bibitem{grubb2015}
G.~Grubb, \emph{Fractional {L}aplacians on domains, a development of
  {H}{\"o}rmander's theory of {$\mu$}-transmission pseudodifferential
  operators}, Adv. Math. \textbf{268} (2015), 478--528. \MR{3276603}

\bibitem{guan2017}
Q.~Guan and M.~Gunzburger, \emph{Analysis and approximation of a nonlocal
  obstacle problem}, J. Comput. Appl. Math. \textbf{313} (2017), 102--118.
  \MR{3573229}

\bibitem{kunisch}
M.~Hinterm{\"u}ller, K.~Ito, and K.~Kunisch, \emph{The primal-dual active set
  strategy as a semismooth {N}ewton method}, SIAM Journal on Optimization
  \textbf{13} (2002), no.~3, 865--888.

\bibitem{kikuchi}
N.~Kikuchi and J.~T. Oden, \emph{Contact problems in elasticity: a study of
  variational inequalities and finite element methods}, SIAM Studies in Applied
  Mathematics, vol.~8, Society for Industrial and Applied Mathematics (SIAM),
  Philadelphia, PA, 1988. \MR{961258 (89j:73097)}

\bibitem{kinderlehrer80}
D.~Kinderlehrer and G.~Stampacchia, \emph{An introduction to variational
  inequalities and their applications}, Pure and Applied Mathematics, vol.~88,
  Academic Press, Inc. [Harcourt Brace Jovanovich, Publishers], New
  York-London, 1980. \MR{567696}

\bibitem{lewy_1969}
H.~Lewy and G.~Stampacchia, \emph{On the regularity of the solution of a
  variational inequality}, Comm. Pure Appl. Math. \textbf{22} (1969), 153--188.
  \MR{0247551}

\bibitem{merton1976}
R.~C. Merton, \emph{Option pricing when underlying stock returns are
  discontinuous}, Journal of Financial Economics \textbf{3} (1976), no.~1, 125
  -- 144.

\bibitem{musina2017}
R.~Musina, A.~I. Nazarov, and K.~Sreenadh, \emph{Variational inequalities for
  the fractional {L}aplacian}, Potential Anal. \textbf{46} (2017), no.~3,
  485--498. \MR{3630405}

\bibitem{ros_oton2014}
X.~Ros-Oton and J.~Serra, \emph{The {D}irichlet problem for the fractional
  {L}aplacian: regularity up to the boundary}, J. Math. Pures Appl. (9)
  \textbf{101} (2014), no.~3, 275--302. \MR{3168912}

\bibitem{rosasco2010}
L.~Rosasco, M.~Belkin, and E.~De~Vito, \emph{On learning with integral
  operators}, J. Mach. Learn. Res. \textbf{11} (2010), 905--934. \MR{2600634}

\bibitem{servadei}
R.~Servadei and E.~Valdinoci, \emph{Lewy-{S}tampacchia type estimates for
  variational inequalities driven by (non)local operators}, Rev. Mat. Iberoam.
  \textbf{29} (2013), no.~3, 1091--1126. \MR{3090147}

\bibitem{silling2000}
S.~A. Silling, \emph{Reformulation of elasticity theory for discontinuities and
  long-range forces}, J. Mech. Phys. Solids \textbf{48} (2000), no.~1,
  175--209. \MR{1727557}

\bibitem{silvestre2007}
L.~Silvestre, \emph{Regularity of the obstacle problem for a fractional power
  of the {L}aplace operator}, Comm. Pure Appl. Math. \textbf{60} (2007), no.~1,
  67--112. \MR{2270163}

\bibitem{tian2015}
H.~Tian, L.~Ju, and Q.~Du, \emph{Nonlocal convection-diffusion problems and
  finite element approximations}, Comput. Methods Appl. Mech. Engrg.
  \textbf{289} (2015), 60--78. \MR{3327145}

\bibitem{visik1967}
M.~I. Vi\v{s}ik and G.~I. \`Eskin, \emph{Elliptic convolution equations in a
  bounded region and their applications}, Uspehi Mat. Nauk \textbf{22} (1967),
  no.~1 (133), 15--76. \MR{0214910}

\bibitem{vollmann2017}
C.~{Vollmann} and V.~{Schulz}, \emph{{Exploiting multilevel Toeplitz structures
  in high dimensional nonlocal diffusion}}, ArXiv e-prints (2017).

\bibitem{Woh00a}
B.~I. Wohlmuth, \emph{A mortar finite element method using dual spaces for the
  {L}agrange multiplier}, SIAM J. Numer. Anal. \textbf{38} (2000), no.~3,
  989--1012. \MR{1781212}

\end{thebibliography}
\end{document}